\documentclass[11pt]{amsart}
\usepackage{amsfonts, amssymb, latexsym, epsfig} 

\usepackage{amscd,amssymb, mathtools, etoolbox}
\usepackage{verbatim}
\usepackage[mathscr]{euscript}
\usepackage{amsmath}
\input xy
\xyoption{all}

\newcommand{\bigcupdot}{\charfusion[\mathop]{\bigcup}{\cdot}}

\makeatletter
\def\moverlay{\mathpalette\mov@rlay}
\def\mov@rlay#1#2{\leavevmode\vtop{%
   \baselineskip\z@skip \lineskiplimit-\maxdimen
   \ialign{\hfil$\m@th#1##$\hfil\cr#2\crcr}}}
\newcommand{\charfusion}[3][\mathord]{
    #1{\ifx#1\mathop\vphantom{#2}\fi
        \mathpalette\mov@rlay{#2\cr#3}
      }
    \ifx#1\mathop\expandafter\displaylimits\fi}
\makeatother

\newsymbol\pp 1275

\usepackage{tikz, standalone}
\usetikzlibrary{shapes,arrows}

\usetikzlibrary{matrix,arrows,backgrounds,shapes.misc,shapes.geometric,patterns,calc,positioning}
\usetikzlibrary{intersections,patterns,arrows,chains,matrix,positioning,scopes,decorations.pathreplacing,decorations.pathmorphing,calc}

\usepackage[colorlinks=true, pdfstartview=FitV, linkcolor=blue, citecolor=blue, urlcolor=blue]{hyperref}
\usepackage{tikz, standalone, float}
\usepackage{fullpage,subcaption}
\usepackage{graphicx}
\usepackage{enumerate}
\def\VR{\kern-\arraycolsep\strut\vrule &\kern-\arraycolsep}
\def\vr{\kern-\arraycolsep & \kern-\arraycolsep}

\usepackage{tikz-cd}
\usepackage{extarrows}

\pgfdeclarelayer{edgelayer}
\pgfdeclarelayer{nodelayer}
\pgfdeclarelayer{background}
\pgfdeclarelayer{foreground}
\pgfsetlayers{background,foreground,edgelayer,nodelayer,main}

\usepackage{blkarray}

\newcommand{\be}{\begin{enumerate}}

\newtheorem{theorem}{Theorem}
\newtheorem*{theoremnonum}{Theorem}
\newtheorem*{lemmanonum}{Lemma}
\newtheorem*{propnonum}{Proposition}
\newtheorem{prop}{Proposition}
\newtheorem{lemma}[prop]{Lemma}
\newtheorem{corollary}[prop]{Corollary}

\theoremstyle{definition}
\newtheorem{definition}[prop]{Definition}

\newtheorem{rmk}{Remark}

\newtheorem{obs}{Observation}

\newtheorem{ex}{Example}
\newenvironment{example}[1][]{\begin{ex}[#1]\pushQED{\qed}}{\popQED \end{ex}}

\makeatletter
\tikzset{join/.code=\tikzset{after node path={%
\ifx\tikzchainprevious\pgfutil@empty\else(\tikzchainprevious)%
edge[every join]#1(\tikzchaincurrent)\fi}}}
\makeatother
\tikzset{>=stealth',every on chain/.append style={join},
         every join/.style={->}}
\tikzstyle{labeled}=[execute at begin node=$\scriptstyle,
   execute at end node=$]
   
   \makeatletter
\tikzset{join/.code=\tikzset{after node path={%
\ifx\tikzchainprevious\pgfutil@empty\else(\tikzchainprevious)%
edge[every join]#1(\tikzchaincurrent)\fi}}}
\makeatother

\begin{document}
\title{Interleaving Distance as a Limit}

\author{Killian Meehan, David Meyer}
%\address{University of Missouri-Columbia, Mathematics Department, Columbia, MO, USA}
%\email[Calin Chindris]{chindrisc@missouri.edu}

%\author{Daniel Kline}
%\address{University of Missouri-Columbia, Mathematics Department, Columbia, MO, USA}
%\email[Daniel Kline]{dbkfz9@mail.missouri.edu}

%\date{\today}
%\bibliographystyle{plain}
%\subjclass[2000]{???, ???}
%\keywords{???}

\maketitle
\setcounter{tocdepth}{1}
%\tableofcontents

\begin{abstract}
Persistent homology is a way of determining the topological properties of a data set.  It is well known that each persistence module admits the structure of a representation of a finite totally ordered set.  In previous work, the authors proved an analogue of the isometry theorem of Bauer and Lesnick for representations of a certain class of finite posets.  The isometry was between  the interleaving metric of Bubenik, de Silva and Scott and  a bottleneck metric which incorporated algebraic information.  The key step in both isometry theorems was proving a matching theorem, that an interleaving gives rise to a matching of the same height.  In this paper we continue this work, restricting to those posets which arise from data while making more general the choice of metrics.  We first show that while an interleaving always produces a matching, for an arbitrary choice of weights it will not produce one of the same height.  We then show that although the matching theorem fails in this sense, one obtains a "shifted" matching (of the correct height) from an interleaving by enlarging the category.  We then prove an isometry theorem on this extended category.  As an application, we make precise the way in which representations of finite partially ordered sets approximate persistence modules.  Specifically, given two finite point clouds of data, we associate a generalized sequence (net) of algebras over which the persistence modules for both data sets can be compared.  We recover the classical interleaving distance uniformly by taking limits.
\end{abstract}
\section{Introduction}
%%%%%%%%%%%%%%%%%%%%%%%%%%%%%%%%%%%%%
%%%%% CHANGES MADE TO PRELIMINARIES, RES/INFL, BEGINNING OF SHIFT ISOMETRY, AND LIMIT SECTIONS%%%%%%%%%%%%%%%%%%%%%%%%%%%%%%%%%%%%%%%%%%%%%%%%%%%%%%%%%%%%%%%%%%%%%%%%%%
\subsection{Persistent Homology}
\noindent
Informally, a \emph{generalized persistence module} is a representation of a poset $P$ with values in a category $\mathcal{D}$.  More precisely, if $\mathcal{D}$ is a category, a generalized persistence module $M$ with values in $\mathcal{D}$ assigns an object $M(x)$ of $\mathcal{D}$ for each $x \in P$, and a morphism $M(x \leq y)$ in ${Mor}_{\mathcal{D}}(M(x),M(y))$ for each $x, y \in P$ with $x \leq y$ satisfying
$$M(x \leq z)=M(y \leq z) \circ M(x \leq y) \text{ whenever }x, y , z \in P \text{ with }x \leq y \leq z.$$  

Perhaps surprisingly, the study of such objects is useful in topological data analysis.  \emph{Persistent homology} uses  generalized persistence modules to attempt to discern the topological properties of a finite data set.  We briefly summarize the algorithm applied to a point cloud of data in the persistent homology setting.  This will lead to \emph{one-dimensional (generalized) persistence modules}, where the poset $P = (0,\infty)$ or $\mathbb{R}$. 

Suppose, for example, we wish to decide whether a data set $D \subseteq {\mathbb{R}}^2$ should be more correctly interpreted as an annulus or a disk.  In order to decide between the two candidates, one calculates the homology of a filtration of simplicial complexes associated to the data set.  This uses the Vietoris-Rips complex ${(C_{\epsilon})}_{\epsilon > 0}$.  Specifically, for each $\epsilon > 0$, we let $C_{\epsilon}$ be the abstract simplicial complex whose $k$-simplices are determined by data points $x_1, x_2, ... x_{k+1}\in D$ where $d(x_i, x_j) \leq \epsilon$ for all $1 \leq i, j \leq k+1$.  Clearly, for ${\sigma} \leq {\tau}$ in $(0, \infty)$,  there is an inclusion of simplicial complexes $C_{\sigma} \hookrightarrow C_{\tau}$, thus we obtain a filtration of simplicial complexes indexed by $(0,\infty)$.  Therefore, the assignment $F:\epsilon \to C_{\epsilon}$ is a representation of the poset $(0, \infty)$ taking values in $Simp$, the category of abstract simplicial complexes.  That is to say, $F$ is a generalized persistence module for $P =(0,\infty)$ and $\mathcal{D} = Simp$.  Since we wish to distinguish between an annulus and a disk, we apply the first homology functor $H_1(-,K)$ to $F$ (where $K$ is some field), to obtain the representation of $P$ with values in $K$-mod, $\epsilon \to H_1(C_{\epsilon},K)$.

Thus, the assigment $H_1(-,K) \circ F \textrm{ given by } \epsilon \to H_1(C_{\epsilon},K)$ is a one-dimensional persistence module.  As $\epsilon$ increases generators for $H_1$ are born and die, as cycles appear and become boundaries.  In persistent homology, one takes the viewpoint that true topological features of the data set can be distinguished from noise by looking for generators of homology which "persist" for a long period of time.  Informally, one "keeps" an indecomposable summands of the module $H_1(-,K) \circ F$ when it corresponds to a wide interval.  Conversely, cycles which disappear quickly after their appearance (narrow ones) are interpreted as noise and disregarded.  

This technique has been widely successful in topological data analysis (see, for example, \cite{carlsson_top}, \cite{stability}, \cite{chazal}, \cite{ghrist}, \cite{carlsson_local}, \cite{applied_1}, \cite{applied_2}, and \cite{applied_3}).  Typically, the category of persistence modules with values in $K$-mod is given a metric-like structure.  So-called \emph{soft stability theorems}, which involve the continuity from the data to the persistence module, have been proven.  Philosophically, these results have established the utility of this method from the perspective of data analysis  (See, for example, \cite{stability}). \emph{Hard stability theorems}, on the other hand, involve comparisons of different metrics on the category of persistence modules.

%%%%%%%%%%%%%%%%%%%%%%%%%%%%%%%%%%%%%
\subsection{Algebraic Stability}
One special type of hard stability theorem is an \emph{algebraic stability theorem}.  In such a theorem, one endows a collection of generalized persistence modules with two metric structures, and  an automorphism $J$ is shown to be a contraction or an isometry.   Of particular interest is the case when $J$ is the identity function and the metrics are an \emph{interleaving metric} and a \emph{bottleneck metric}.  Algebraic stability theorems of this type are common (see \cite{lesnick}, \cite{zigzag}, \cite{induced_matchings}, and \cite{carlsson}).  While in the literature, the word "interleaving" is frequently used to describe slightly different metrics, we believe that the interleaving metric suggested by Bubenik, de Silva and Scott (see \cite{bubenik}) has the advantage of being both most general, and categorical in nature.  This interleaving metric makes sense on any poset $P$, and reduces to the interleaving metric of \cite{induced_matchings} when $P = (0,\infty)$.  Alternatively, a bottleneck metric is nothing more than a way of extending a metric defined on a set ${\Sigma}$, to the collection of all ${\mathbb{Z}}_{\geq 0}$-valued functions with finite support on ${\Sigma}$.  In this context, this is applied to the decomposition of a generalized persistence module into its indecomposable summands with their corresponding multiplicities.

\subsection{Connections to Finite-dimensional Algebras}
This paper concerns algebraic stability studied using techniques from the representation theory of algebras.  Such representations appear because one-dimensional persistence modules arising from data always admit the structure of a representation of a finite totally ordered set.  This fact comes from the simple observation that the generalized persistence module given by $F:\epsilon \to C_{\epsilon}$ is necessarily a step function.  More precisely, let 
$$P_n = \{\epsilon_1 < \epsilon_2 <... < \epsilon_n\}  = \{ \epsilon \in (0,\infty): C_{\epsilon} \neq \lim \limits_{\tau \to {{\epsilon}^-}} C_{\tau} \}.$$

By definition, $F$ is constant on all intervals of the form $[\epsilon_i,\epsilon_{i+1})$.  Thus, clearly, both $F$ and $H_1(-,K) \circ F$ admit the structure of a generalized persistence module for $P= P_n$.  %This identification corresponds to $\delta$ is our schematic below.
%\vspace{.2 in}
\begin{comment}
%%%%%%%%%%%%%%%%%%%%%%%%%%%%%%%%%%%%%%
\tikzstyle{block} = [rectangle, draw,fill=white,text centered, rounded corners]
\tikzstyle{colored1} = [rectangle, draw, fill=blue!20,text centered, rounded corners, minimum height=4em]

%%%%%%%%%%----------cloud------------------------------

\begin{tikzpicture}[node distance =4cm, auto]

    % Place nodes
\node [cloud, draw,cloud puffs=10,cloud puff arc=120, aspect=2, inner ysep=1em] (data) {Data} ;
    \node [block, above right of=data,align=center,yshift=-.75cm](comp) {$P=(0,\infty),$\\ $\mathcal{D}=Simp$};
    \node [colored1, right of=comp,align=center] (module) {$P = (0,\infty),$\\ $\mathcal{D} = K$-mod};
    \node [colored1, right of=module,align=center] (rank) {$P = (0,\infty),$\\ $\mathcal{D} = K$-mod};

    \path [line, dashed](data) -- node {$F$} (comp);
    \path [line, dashed] (comp) -- node {$H_1(-,K)$} (module);
    \path [line] (module) --  node {$J$} (rank);

\node [block, below right of=data,align=center,yshift=+.75cm](comp2) {$P = P_n,$\\ $\mathcal{D} = Simp$};
\node [colored1, right of=comp2,align=center] (module2) {$P = P_n,$\\ $\mathcal{D} =K$-mod};
\node [colored1, right of=module2,align=center] (rank2) {$P = P_n,$\\ $\mathcal{D} =K$-mod};

\path[line,dashed](comp2) -- node {$H_1(-,K)$} (module2);
\path[line](module2) -- node {$J$} (rank2);

\draw[->>] (comp) -- node {$\delta$} (comp2);
\draw[->>] (module) -- node {$\delta$} (module2);
\draw[->>] (rank) -- node{$\delta$} (rank2);
\end{tikzpicture}

\vspace{.2 in}
\end{comment}
%%%%%%%%%%%%%%%%%%%%%%%%%%%%%%%%%
When we restrict the structure of a one-dimensional persistence module to $P_n$, we say informally that we are \emph{discretizing}.  In this sense, generalized persistence modules for finite totally ordered sets are the discrete analogue of one-dimensional persistence modules.

While the above description of persistent homology will always discretize to a generalized persistence module for a finite totally ordered set, representations of many other infinite families of finite posets also have a physical interpretation in the literature (see \cite{zigzag}, \cite{carlsson}, \cite{ladder}).  For example, multi-dimensional persistence modules (see \cite{carlsson}) will discretize in an analagous fashion to representations of a different family of finite posets.  This is relevant because there is a categorical equivalence between the generalized persistence modules for a finite poset $P$ with values in $K$-mod, and the module category of the finite-dimensional $K$-algebra $A(P)$, the poset (or incidence) algebra of $P$.  In \cite{meehan_meyer_1}, we suggested a template for a representation-theoretic algebraic stability theorem and proved the following theorem. 

\vspace{.1 in}
\noindent
\begin{theoremnonum}(Theorem 1 \cite{meehan_meyer_1})
\label{main}
Let $P$ be an $n$-Vee and let $\mathcal{C}$ be the full subcategory of $A(P)$-modules consisting of direct sums of convex modules.  Let $(a,b) \in {\mathbb{N}} \times {\mathbb{N}}$ be a weight and let $D$ denote interleaving distance (corresponding to the weight $(a,b)$) restricted to $\mathcal{C}$.  \\Let  $W(M) = \textrm{min}\{ \epsilon: \textrm{Hom}(M, M\Gamma \Lambda) = 0, \Gamma, \Lambda \in  \mathcal{T}(\mathcal{P}), h(\Gamma), h(\Lambda) \leq \epsilon \}$, and let $D_B$ be the bottleneck distance corresponding to the interleaving distance and $W$ on $\mathcal{C}$.  Then, the identity is an isometry from $$(\mathcal{C},D) \xrightarrow{Id} (\mathcal{C},D_B).$$
\end{theoremnonum}

This result established an algebraic version of the isometry theorem for a large class of posets.  Our choice of interleaving metric was the metric of Bubenik, de Silva and Scott with a slight modification.  The bottleneck metric incorporates algebraic information, but corresponds to the bottleneck metric used by Bauer and Lesnick when applied to one-dimensional persistence modules.  In the above, $1$-Vees are exactly finite totally ordered sets, and in the case $P$ is a $1$-Vee, the category $\mathcal{C}$ is the full module category.  Once some parameters are chosen, the collection of interleavings between two modules has the structure of an affine variety.  The interleaving distance corresponded to the smallest parameter for which the corresponding variety was nonempty.

While this theorem covered a much more general class of posets than finite totally ordered sets (many of which had wild representation type), we restricted to metrics coming from a so-called "democratic" choice of weights on the Hasse quiver of the poset.  We now wish to extend this to metrics coming from arbitrary weights in the case of finite totally ordered sets.  

In addition, from the perspective of data analysis, in \cite{meehan_meyer_1} we pointed out two potential issues which arise when one passes to the jump discontinuities of the Vietoris-Rips complex.  This paper will address these issues.  The first issue is that a finite data set $D$ produces not only a generalized persistence module, but also to its accompanying algebra.  Because of this, a priori two persistence modules may not be able to be compared simply because they are not modules for the same algebra.  The second issue is that information about the width of the interval $[\epsilon_i, \epsilon_{i+1})$ should be considered, but appears lost when one discretizes.  Here we show that both of these concerns can be dealt with successfully after we establish an algebraic stability theorem for $P_n$ with an arbitrary choice of weights.

\subsection{Main Results}

In this paper, we restrict our attention to the class of posets which arise when one-dimensional persistence modules are discretized, but extend our analysis to interleaving metrics coming from an arbitrary choice of weights.  In this situation, we show that an interleaving need not produce a matching of the same height.  Indeed, when one makes the most natural choice of weights corresponding to the geometry of the data set, this is almost certain to be the case.  Even in this situation, however, we obtain a "shifted" isometry theorem by extending to an algebra with a larger module category.  This result is stated in the first main theorem.

\begin{theorem}[Shift Isometry %prev 'matching'
 Theorem]
\label{shift_isometry_theorem}
Let $X_1\subseteq \mathbb{R}$ be finite and $Sh(X_1)$ be its shift refinement. Let $P_{X_1}$ and $P_{Sh({X_1})}$ be the corresponding totally ordered sets.  Let $C$ be the full subcategory of $A(P_{Sh(X_1)})$-mod given by $\mathcal{C}=im(j(X_1,Sh(X_1)))$ and let $D^{Sh(X_1)}$, $D_B^{Sh(X_1)}$ be the interleaving metric and bottleneck metric respectively. Let $j =  j(X_1,Sh(X_1))$, and $Id$ denote the identity.  Then, the horizontal arrow is an isometry and the diagonal arrows are contractions.

\begin{center}
\includegraphics{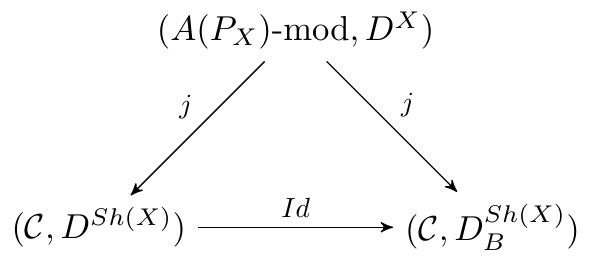}
%\begin{tikzpicture}[commutative diagrams/every diagram]
%	\matrix[matrix of math nodes, name=m, commutative diagrams/every cell,row sep=.7cm,column sep=.45cm] {
%	 \pgftransformscale{0.2}
%		{} & {} & {}& (\mathcal{C},D^{Sh(X_1)})\\
%		(A(P_{X_1})\textrm{-mod},D^{X_1}) & {}& {} & {}\\
%		{} & {} & {} & (\mathcal{C},D_B^{Sh(X_1)})\\ };
		
%		\path[commutative diagrams/.cd, every arrow, every label]
%		(m-2-1) edge node[above] {$j$} (m-1-4)
%		(m-2-1) edge node[above] {$j$} (m-3-4)
%		(m-1-4) edge node[right] {$Id$} (m-3-4);

%\end{tikzpicture}
\end{center}
\end{theorem}
Theorem \ref{shift_isometry_theorem} says that in particular, by using the functor $j$ to enlarge the category, one obtains the shifted isometry 
$$(\mathcal{C},D^{Sh(X_1)})\xrightarrow{Id} (\mathcal{C},D_B^{Sh(X_1)}).$$ Moreover, $j$ is fully faithful and is itself a contraction.  When the choice of weights is general, this is a natural extension of the classical isometry theorem.  We then use the shift isometry theorem to prove our second main result.

\begin{theorem}
\label{main theorem}
Let $I, M$ be persistence modules (for $\mathbb{R}$ ) whose endpoints lie in $L$.  Then,
$$\displaystyle\lim\limits_{X \in \Delta(D)}\big(D^X(\delta^XI,\delta^XM)\big) =  D(I,M). $$
\end{theorem}

Theorem \ref{main theorem} says that, again, with the most natural choice of weights, our modification of the interleaving metric of Bubenik, de Silva and Scott approximates the classical interleaving metric in the sense that one can take a limit over a directed set to obtain the classical distance. This limit is uniform in that the choice of refinement is independent of the modules to be compared. In addition, by the squeeze theorem the discrete bottleneck metric converges to the classical bottleneck metric (and therefore the interleaving metric) as well.

This paper will be organized as follows: in Section \ref{section preliminaries} we give a brief summary of background information, then in Section \ref{section rest} we discuss restriction and inflation.  In the next two sections we prove our two main results.  In Section \ref{section shift} we prove Theorem \ref{shift_isometry_theorem}, and in Section \ref{section interleaving} we prove Theorem \ref{main theorem}.  Then, in Section \ref{section regu} we give an analysis of posets in which the original matching theorem fails.

\section{Acknowledgements}
The authors wish to acknowledge Calin Chindris both for introducing us to this field of study, and for all of his guidance.  K. Meehan was supported by the NSA under grant H98230-15-1-0022.

\section{Preliminaries}
\label{section preliminaries}

\subsection{Generalized Persistence Modules}
%%%%%%%%%%%%%%%%%%%%%%%%%%%%%%%%%%%%%%%
Recall that if $P$ is a poset and $\mathcal{D}$ is a category, a generalized persistence module $M$ with values in $\mathcal{D}$ assigns an object $M(x)$ of $\mathcal{D}$ for each $x \in P$, and a morphism $M(x \leq y)$ in ${Mor}_{\mathcal{D}}(M(x),M(y))$ for each $x, y \in P$ with $x \leq y$ that satisfies
 $$M(x \leq z)=M(y \leq z) \circ M(x \leq y) \textrm{ whenever }x, y , z \in P \textrm{ and }x \leq y \leq z.$$ 

Let ${\mathcal{D}}^P$ denote the collection of generalized persistence modules for $P$ with values in $\mathcal{D}$.  If $F, G \in {\mathcal{D}}^P$, a morphism from $F$ to $G$ is a collection of morphisms $\{\phi(x)\}$, with $\phi(x) \in {Mor}_{\mathcal{D}}(F(x),G(x))$ for all $x \in P$, such that for all $x \leq y$ we have a commutative diagram below for each $x \leq y$ in $P$.

\begin{center}

\begin{tikzpicture}[commutative diagrams/every diagram]
	\matrix[matrix of math nodes, name=m, commutative diagrams/every cell,row sep=.7cm,column sep=1cm] {
	 \pgftransformscale{0.2}
		F(x) & F(y)\\
		G(x) & G(y)\\ };
		
		\path[commutative diagrams/.cd, every arrow, every label]
			(m-1-1) edge node {$F(x \leq y)$} (m-1-2)
			(m-1-1) edge node[swap] {$\phi(x)$} (m-2-1)
			(m-2-1) edge node {$G(x \leq y)$} (m-2-2)
			(m-1-2) edge node {$\phi(y)$} (m-2-2);
		
\end{tikzpicture}

\end{center}
With such morphisms, ${\mathcal{D}}^P$  is a category.  Equivalently, one could regard the poset $P$ as a thin category.  Then, a generalized persistence module will correspond to a covariant functor from $P$ to $\mathcal{D}$, and morphisms in ${\mathcal{D}}^P$ will be natural transformations.  In this paper $P$ will always be a finite totally ordered set, $(0,\infty), \textrm{ or }\mathbb{R}$.  $\mathcal{D}$ will be either $Simp$ or $K$-mod.
%%%%%%%%%%%%%%%%%%%%%%%%%%%%%%%%%%%%%%%

\subsection{Representation Theory of Algebras}
\label{algebra}

In this subsection we give a brief summary of $K$-algebras and their representations (modules).  For a more expansive introduction, see \cite{auslander}, \cite{benson1}, \cite{benson2}.  For more on the connection between finite-dimensional algebras and generalized persistence modules see \cite{oudot} or \cite{meehan_meyer_1}.  Throughout, let $K$ denote a field.  If $R$ is a $K$-algebra, by an $R$-module, we mean a finite-dimensional, unital, left $R$-module.  The category $R$-mod consists of $R$-modules together with $R$-module homomorphisms.

Recall that an $R$-module $M$ is \emph{indecomposable} if it is not isomorphic to a direct sum of two of its proper submodules.  The category $R$-mod is an abelian Krull-Schmidt category.  That is, every module can be written as a direct sum of indecomposable modules in a unique way up to order and isomorphism.  Moreover, the decomposition of modules is compatible with homomorphisms in the following sense.

\begin{prop}
Let $R$ be a $K$-algebra, and let $M, N$ be $R$-modules.  Say, $M \cong \oplus M_i$ and $N \cong \oplus N_j$.  Then, as vector spaces,
$$\textrm{Hom}(M,N) \cong \bigoplus\limits_{i,j} \textrm{Hom}(M_i,N_j). $$
\end{prop} 

This says that any module homomorphism $f:M \to N$ can be factored into a matrix of module homorphisms $f^i_j : M_i \to N_j$.

\subsubsection{Quivers and their Representations}
We now define quivers and their representations.

\begin{definition}
A quiver $Q = (Q_0,Q_1, t, h)$ is an ordered tuple, where $Q_0, Q_1$ are disjoint sets, and $t, h : Q_1 \to Q_0$.
\end{definition}

We call elements of $Q_0$ vertices, and elements of $Q_1$ arrows.  The functions $t$ and $h$ denote the tail (start) and head (end) of the arrows.  Thus, clearly $Q$ is exactly a directed set.  We will always suppose the sets $Q_0, Q_1$ are finite.  

\begin{ex} 
\label{quivers}
Below are two quivers.

\begin{figure}[H]
\centering
\begin{subfigure}[b]{0.5\textwidth}
\label{quiver1}
\begin{tikzpicture}[normal line/.style={->},shorten >=1pt,font={\it\small},node distance=2cm,main node/.style={circle,scale=.5,fill=blue!10,draw,font=\sffamily\small\bfseries}]
 
  \node (1) {$\bullet_1$};
  \node (2) [right of=1] {$\bullet_2$};
  \node (3) [right of=2] {$\bullet_3$};
  \node (4) [above right of=3] {$\bullet_4$};
  \node (5) [below right of=3] {$\bullet_5$};

  \path[normal line]
    (1) edge node [above] {a} (2) 
    (2) edge node [above] {b} (3)
    (3) edge node [above] {c} (4)
    (3) edge node [above] {d} (5);

\end{tikzpicture}
\caption*{A}

\end{subfigure}
~
\begin{subfigure}[b]{0.5\textwidth}
\label{quiver2}
\centering
\begin{tikzpicture}[normal line/.style={->},shorten >=1pt,font={\it\small},node distance=3cm,main node/.style={circle,scale=.5,fill=blue!10,draw,font=\sffamily\small\bfseries}]
 
  \node (1) {$\bullet_1$};
  \node (2) [above right of=1] {$\bullet_2$};
  \node (3) [right of=1] {$\bullet_3$};
  \node (4) [right of=3] {$\bullet_4$};

  \path[normal line]
  	(1) edge node [right] {a} (2) 
	(2) edge [bend right=30] node [right] {c} (3)
	(3) edge node [below right] {b} (1)
	(3) edge [bend right=45] node [right] {d} (2)
	(3) edge [bend right=90] node [below] {e} (4)
	(3) edge [bend right] node [above] {f} (4)
	(4) edge [loop] node [above] {g} (4);

\end{tikzpicture}
\caption*{B}
\label{quiver2}

\end{subfigure}
\end{figure}
\noindent
Quiver A corresponds to $Q_0 = \{1,2,3,4,5\}, Q_1 = \{a,b,c,d\}$, for an appropriate choice of the functions $h, t$.  Similarly, quiver B corresponds to the sets $Q_0 = \{1, 2, 3, 4\}$, and $Q_1 = \{a, b, c, d, e, f, g\}$. 
\end{ex}

A \emph{path} is a sequence of arrows $p = a_1 ... a_n$ where $t(a_i) = h(a_{i+1})$.  The length of the path is the number of terms in the sequence $p$.  In addition, at each vertex $i$ there is a "lazy" path $e_i$ of length $0$ at the vertex $i$.  We extend the functions $h, t$ to paths, by defining $t(p) = t(a_n)$ and $h(p) = h(a_1)$.   In addition, $t(e_i) = h(e_i) = i$.  An \emph{oriented cycle} is a path $p$ of length greater than or equal to one with $t(p) = h(p)$.

\begin{definition}
A representation $V$ of a quiver $Q$ is a family $V = (\{V(i)\}_{i \in Q_0}, \{V(a)\}_{a \in Q_1})$, where $V(i)$ is a $K$-vector space for every $i \in Q_0$, and $V(a) : V(t(a)) \to V(h(a))$ is a $K$- linear map for every $a \in Q_1$.
\end{definition}

For a fixed quiver $Q$ and field $K$, the collection of all representations of $Q$ is a category with morphisms given as follows.  If $V$, $W$ be representations of $Q$,  \emph{morphism} from $V$ to $W$, $\phi: V \to W$ is a collection of linear maps $\{\phi(i)\}_{i \in Q_0}$ with $\phi(i):V(i) \to W(i)$ such that the diagam below commutes for all $a \in Q_1$

\begin{center}

\begin{tikzpicture}[commutative diagrams/every diagram]
	\matrix[matrix of math nodes, name=m, commutative diagrams/every cell,row sep=.7cm,column sep=1cm] {
	 \pgftransformscale{0.2}
		V(t(a)) & V(h(a))\\
		W(t(a)) & W(h(a))\\ };
		
		\path[commutative diagrams/.cd, every arrow, every label]
			(m-1-1) edge node {$V(a)$} (m-1-2)
			(m-1-1) edge node[swap] {$\phi(t(a))$} (m-2-1)
			(m-2-1) edge node[swap] {$V(a)$} (m-2-2)
			(m-1-2) edge node {$\phi(h(a))$} (m-2-2);
		
\end{tikzpicture}

\end{center}

\medskip
We denote by $Rep(Q)$, the category of $K$-representations of the quiver $Q$.  When $\phi: V \to W$ is a morphsim from $V$ to $W$ and $\phi(i)$ is invertible for all $i$, then we say $\phi$ is an isomorphism.  If this is the case, we say that $V$ and $W$ are isomorphic. 
\begin{definition}
\label{support}
If $V$ is a representation of a quiver $Q$ we say the support of $V$, Supp$(V)$ is the set of all vertices $i \in Q_0$, such that $V(i)$ is not the zero vector space.
\end{definition}

If $Q$ is a quiver, the \emph{path algebra} $KQ$ is the $K$-vector space with basis consisting of all paths (including those of length zero).  Multiplication in $KQ$ is defined as the $K$-linear extension of concatenation of paths.  That is, if $p, q$ are paths, then $p \cdot q = pq$, if $pq$ is a path, and zero otherwise.  With this multiplicative structur, $KQ$ is a $K$-algebra.  When $Q$ has no oriented cycles, $KQ$ is finite-dimensional.  The following proposition shows the connection between quivers and algebras.

\begin{prop}
\label{quiver equivalence}
\label{bound quiver equivalence}
Let $Q$ be a quiver. Then, there exists a natural equivalence between $Rep(Q)$ and $KQ$-mod.
\end{prop}

\subsubsection{The Poset Algebra for $P_n$}

We will now define the poset algebra $A(P_n)$ of the finite poset $P_n$.  For a discussion of the poset algebra for an arbitrary finite poset, see \cite{meehan_meyer_1}.  This discussion fits into the framework of \cite{meehan_meyer_1}, since the finite totally ordered set $P_n$ is exactly a $1$-Vee.  

\begin{definition}
Let $P$ be a finite poset.  Let $Q_P$ be the quiver with $Q_0 = P$.  There is an $a \in Q_1$ with $t(a) = x, h(a) = y$ if,
\begin{enumerate}[(i.)]
\item $x < y$, and 
\item there is no $t \in P$, with $x < t < y$.
\end{enumerate}
\end{definition}

The quiver $Q_P$ is called the Hasse quiver of the poset $P$.  By inspection, the Hasse quiver of $P$ is exactly the lattice of the poset with arrows corresponding to minimal proper relations. Note that if $Q_P$ is the Hasse quiver of the poset $P$ and there is an arrow that begins at a vertex $v$ and ends at the vertex $w$, it is necessarily unique. Thus we may omit labeling arrows.

\begin{ex} The Hasse quiver of $P_n$ is drawn below.
\begin{center}
\includegraphics[scale=2]{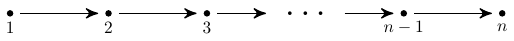}
\end{center}
\end{ex}

This quiver is the so-called equioriented $\mathbb{A}_n$ quiver.

\begin{definition}
Let $P = P_n$ be totally ordered.  Then the incidence algebra $A(P_n)$ of $P_n$ is the quiver algebra of $\mathbb{A}_n$.
\end{definition}

By Proposition \ref{quiver equivalence} we may identify every module for the algebra $A(P_n)$ with a representation of the equioriented $\mathbb{A}_n$ quiver, and vice versa. Thus, from this point forward, we pass freely between generalized persistence modules for $P_n$ with values in $K$-mod, representations of the quiver $\mathbb{A}_n$ and modules for the algebra $A(P_n)$.  

\begin{definition}
\label{convex}
An indecomposable module $M$ for $A(P_n)$ is convex if %its dimension vector consists of only zeros and ones and
it is isomorphic to a module $M'$ where $M'$ satisfies
\begin{enumerate}[(i.)]
\item for all $x,y \in \text{Supp}(M') \text{, with } x \leq y, \text{ the linear map } M'(x \leq y)\text{ is given by } Id_K,$and 
\item for all vertices $x$, the dimension of the $K$-vector space at $M(x)$ is at most $1$.
\end{enumerate}
\end{definition}

Note that the definition of convex modules makes sense for an arbitrary finite poset.  In general, a poset algebra may have infinitely many isomorphism classes of indecomposable modules, but necessarily has only finitely many convex modules (for more on convex modules for poset algebras see \cite{meehan_meyer_1}). In the algebra $A(P_n)$, however, every indecomposable module is convex.  Thus, while in \cite{meehan_meyer_1} we restricted to $\mathcal{C}$, the category generated by convex modules, this in fact corresponds to the full module category for $A(P_n)$.  It is well known that the isomorphism classes of convex modules for $A(P_n)$ are in one-to-one correspondence with the closed intervals in the poset $P_n$.  That is to say, if $x \leq y$ in $P_n$ there is a unique isomorphism class of an indecomposable module represented by $I$ with Supp$(I) = [x,y]$.  Moreover, all indecomposable modules for $A(P_n)$ are of this form up to isomorphism.  If $I$ is isomorphic to the convex module with support given by $[x,y]$, we write $I \sim [x,y]$.

We now conclude with a statement about homomorphisms between convex modules.  Since $P_n$ is a $1$-Vee, we have the following.

\begin{lemmanonum}(see for example Lemma 34 \cite{meehan_meyer_1})
\label{hom2}
Let $P=P_n$ be totally ordered, and let $I, J$ be convex modules.  Then, Hom$(I,J)\cong K$ or $0$ (as a vector space).
\end{lemmanonum}
In addition, when $I, J$ are convex, conditions on their supports determine the cardinality of the vector space Hom$(I,J)$.

\begin{lemmanonum}(see for example Lemma 40 \cite{meehan_meyer_1})
\label{hominterval}
If $I,J$ are convex modules for $P_n$ with $I \sim [x,X], J \sim [y,Y]$.  then Hom$(I,J)\neq0$ if and only if $$y\leq x\leq Y\leq X.$$ 
\end{lemmanonum}
When Hom$(I,J) \neq 0$, let $\chi([x,Y])$ denote the characteristic function on the interval $[x,Y]$.  Then, the linearization of $\chi([x,Y])$, $\Phi_{I,J}$ is an explicit generator for the vector space Hom$(I,J)$.

\subsection{Interleaving Metrics on $P$ and $P^+$}
%%%%%%%%%%%%%%%%%%%%%%%%%%%%%%%%%
We begin with the construction of the interleaving metric of Bubenik, de Silva and Scott (see \cite{bubenik}). In the interest of generality, initially we let $P$ be an arbitrary poset.
\begin{definition}
Let $P$ we a finite poset and $\mathcal{T}(P^-)$ be the collection of endomorphisms of the poset $P$ with the additional property that $\Lambda p \geq p$ for all $p \in P$.  We call the elements of $\mathcal{T}(P^-)$  translations.  
\end{definition}

Explicitly, a function $\Lambda: P \to P$ is an element of $\mathcal{T}(P^-)$ if and only if $x \leq y \implies \Lambda x \leq \Lambda y \textrm{, and }
\Lambda p \geq p\textrm{, for all } p \in P.$  It is easy to see that the set $\mathcal{T}(P^-)$ is itself a poset under the relation $\Lambda \leq \Gamma$ if for all $p \in P$, $\Lambda p \leq \Gamma p$, and that $\mathcal{T}(P^-)$ is a monoid under functional composition.  

We now consider interleaving metric suggested by \cite{bubenik} on generalized persistence modules.  Let $d$ be any metric on a finite poset $P$, we define a height function $h = h(d)$ on $\mathcal{T}(P^-)$.

\begin{definition}
For $\Lambda \in \mathcal{T}(P^-)$ set $h(\Lambda) = \textrm{sup}\{d(x,\Lambda x): x \in P \}$
\end{definition}

Of course, since $P$ is finite, we may replace a supremum with maximum.  Proceeding as in \cite{bubenik}, let $\mathcal{D}$ be any category.  Then, $\mathcal{T}(P^-)$ acts on $\mathcal{D}^P$ on the right by the formulae
$$(F \cdot \Gamma)(p) = F(\Gamma p), \text{ and } (F \cdot \Gamma)(p \leq q) =F(\Gamma p \leq \Gamma q).$$

\begin{definition}

Let $F,G\in\mathcal{D}^P$ and let $\Gamma, \Lambda$ be translations on $P$. A $(\Gamma,\Lambda)$-interleaving between $F$ and $G$ is a pair of morphisms in $\mathcal{D}^P$, $\phi:F\to G\Lambda,\,\,\,\,\psi:G\to F\Gamma$ such that the following diagrams commute.

\begin{center}

\begin{tikzpicture}[commutative diagrams/every diagram]
	\matrix[matrix of math nodes, name=m, commutative diagrams/every cell,row sep=.7cm,column sep=.45cm] {
	 \pgftransformscale{0.2}
		F & & F\Gamma\Lambda  & & F\Gamma & \\
		& G\Lambda  & & G & & G\Lambda \Gamma\\ };
		
		\path[commutative diagrams/.cd, every arrow, every label]
			(m-1-1) edge node {} (m-1-3)
			(m-1-1) edge node[swap] {$\phi$} (m-2-2)
			(m-2-2) edge node[swap] {$\psi \Lambda $} (m-1-3)
			(m-2-4) edge node {$\psi$} (m-1-5)
			(m-2-4) edge node[swap] {} (m-2-6)
			(m-1-5) edge node {$\phi \Gamma$} (m-2-6);
					
\end{tikzpicture}

\end{center}

\end{definition}
The two horizontal maps in the diagram above are given by the formulae
$$\text{for all }p \in P, F( p \leq \Gamma\Lambda p) \text{ and } G(p \leq \Lambda \Gamma p) \text{, respectively.}  $$
Note that two generalized persistence modules are $(1,1)$-interleaved, where $1$ is the identity translation, if and only if they are isomorphic.  
\begin{definition}[\cite{bubenik}]
Let $M,N\in\mathcal{D}^P$. Given any metric $d$ on $P$, we define $D = D(d)$ by the formula;
\begin{eqnarray*}
&&D(M,N) :=\textrm{inf} \{ \epsilon : \exists (\Gamma,\Lambda) \textrm{-interleaving with } {\textrm{sup}}_{p \in P}d(p,\Gamma p), {\textrm{sup}}_{p \in P}d(p,\Lambda p) \leq \epsilon \}\\
&&=\textrm{inf}\{\epsilon: \exists (\Gamma,\Lambda) \textrm{-interleaving with } h(\Lambda), h(\Gamma) \leq \epsilon\}.
\end{eqnarray*}
\end{definition}

From Bubenik, de Silva and Scott (\cite{bubenik}), we know that $D$ is a pseudometric on ${\mathcal{D}}^P$, and for any category $\mathcal{F}$ and functor $R: \mathcal{D} \to \mathcal{F}$, post-composition by $R$ is a contraction from ${\mathcal{D}}^P$ to ${\mathcal{F}}^P$.  With hard stability theorems in mind, the fact that post-composition by any functor induces a contraction is particularly noteworthy.  Still, independent of the choice of metric $d$ on $P$, without modification the resulting metric $D=D(d)$ need not be a proper metric, simply because the collection of translations is not be rich enough to provide interleavings between arbitrary generalized persistence modules.  This is unfortunate, since $\mathcal{T}(P^-)$ is defined naturally for any poset $P$.  The failure comes from the fact that finite posets will always have fixed points. We now restrict our attention to the case where $\mathcal{D}$ is $K$-mod.

We say that $p \in P$ is a \emph{fixed point} of $P$, if $\Lambda p = p$ for all $\Lambda$ in $\mathcal{T}(P^-)$.  If $p$ is a fixed point of $P$ and $dim_K (M(p)) < dim_K (N(p)) $, then the diagram below does not commute for any morphisms $\phi, \psi$ and any translations $\Lambda, \Gamma$, since the composition cannot have full rank as required. Thus, for any choice of metric $d$, $D(M,N) = \infty$.

\begin{center}
\begin{tikzpicture}[commutative diagrams/every diagram]
	\matrix[matrix of math nodes, name=m, commutative diagrams/every cell,row sep=.7cm,column sep=.45cm] {
	 \pgftransformscale{0.2}
		N(p) & & N(\Gamma\Lambda p)=N(p) \\
		& M(\Lambda p)= M(p) & \\ };
		
		\path[commutative diagrams/.cd, every arrow, every label]
			(m-1-1) edge node {$N(p\leq \Gamma\Lambda p)= Id_{N(p)}$} (m-1-3)
			(m-1-1) edge node[swap] {$\phi_p$} (m-2-2)
			(m-2-2) edge node[swap] {$\psi_{\Lambda p}=\psi_p$} (m-1-3);

\end{tikzpicture}
\end{center}

Note that finite posets will always have fixed points (for a more general discussion on fixed points for arbitrary finite posets, see \cite{meehan_meyer_1}). In $P_n$, the maximal element is necessarily a fixed point. In particular, if $p$ is a fixed point of $P$, with $p \in \textrm{Supp}(M), p \notin \textrm{Supp}(N)$, then $D(M,N) = \infty$.  Because of this, there is no hope of realizing any honest metric as an interleaving metric on any finite poset.

With this in mind, we make the following slight modification.  We set $P^+ = P \cup \{\infty \}$ with added relations $p \leq \infty$, for all $p \in P$.  We may now view $A(P)$-mod as the full subcategory of $A(P^+)$-modules where all objects are supported in $P$.  Note that the Hasse quiver for $P^+$ is simply the Hasse quiver for $P$ with added edges connecting maximal elements of $P$ to the element $\infty$. Now there exists an interleaving between any two generalized persistence modules for $P^+$ that are supported in $P$.

We will now define $d$ on ${(P_n)}^+$ by attaching positive weights to each edge of the Hasse quiver of ${(P_n)}^+$. In the so-called \emph{democratic} case below, all edges in $P_n$ are given the same value, while the new edge is given a potentially different value. In the second case the choice of weights is arbitrary.

%%%%%%%%%%%%%%%%%%%%%%%%%%%%%%%%%%%%%

\begin{center}
\includegraphics[scale=2]{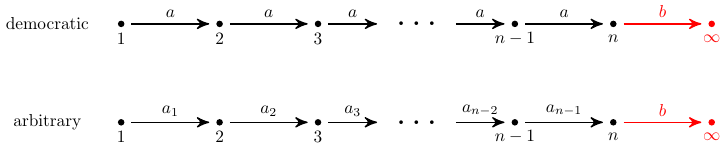}
\end{center}

%%%%%%%%%%%%%%%%%%%%%%%%%%%%%%%%%%%%%%

For an analysis of the democratic case for more general posets, see \cite{meehan_meyer_1}.  In this paper we will investigate the arbitrary (undemocratic) choice of weights for finite totally ordered sets.  This corresponds to the labeling of the Hasse quiver above, where $a_i, b \in (0,\infty)$. When this is the case, we say  $P_n$ has weights given by $(\{a_i\},b)$.  
\begin{definition}
Now, we let $d_{(\{a_i\},b)}$ denote the weighted graph metric on the Hasse quiver of ${(P_n)}^+$, and let $\mathcal{D}$ be a category.  Then, $D = D(d_{(\{a_i\},b)})$ is the interleaving metric corresponding to the weights $d_{(\{a_i\},b)}$ on  ${\mathcal{D}}^{(P_n)}$.
\end{definition} 
Now $D$ defines the structure of a finite metric space on the isomorphism classes of elements of ${\mathcal{D}}^{(P_n)}$.  We will now write $\mathcal{T}(P_n)$ for $\mathcal{T}(({(P_n)}^+)^-)$, and from this point forward, we only consider $\mathcal{D} = K$-mod, and $P = P_n$ suspended at infinity. By Proposition \ref{quiver equivalence}, we write $A(P_n)$-mod for $\mathcal{D}^P$.  Since $P_n$ is a $1$-Vee, we obtain the following two useful lemmas.

%%%% This has to go after INTERLEAVINGS because we talk about translations.
\begin{lemmanonum}(Lemma 22 \cite{meehan_meyer_1})
\label{T(P)}
Consider $P_n$ with weights given by $({\big\{a_i\big\}}_{i=1}^n,b)$, with $a_i, b >0$ in $\mathbb{R}$.  Let $d$ denote the corresponding weighted graph metric on the Hasse quiver of $P^+$.  Then,
\begin{enumerate}[(i.)]
\item For each $\epsilon \in \{ h(\Lambda): \Lambda  \in \mathcal{T}(P)\}$, the set $\{ \Gamma \in \mathcal{T}(P): h(\Gamma) = \epsilon \} $ has a unique maximal element ${\Lambda}_{\epsilon}$.
\item The set $\{ {\Lambda}_{\epsilon} \}$ is totally ordered, and ${\Lambda}_{\epsilon} \leq {\Lambda}_{\delta}$ if and only if $\epsilon \leq \delta$.
\item If $\Lambda, \Gamma \in \mathcal{T}(P)$ with $h(\Lambda), h(\Gamma) \leq \epsilon$ then there exists a ${\Lambda}_{\delta}$ with $\Lambda, \Gamma \leq {\Lambda}_{\delta}$, and $h({\Lambda}_{\delta})= \delta = \textrm{ max} \{ h(\Lambda), h(\Gamma) \}$.
\end{enumerate}

\end{lemmanonum}

The above lemma justifies the passage from a $(\Lambda,\Gamma)$-interleaving to a $(\Lambda_\epsilon,\Lambda_\epsilon)$-interleaving where $\epsilon$ is the maximum of $h(\Lambda)$ and $h(\Gamma)$.  Moreover, the monoid $\mathcal{T}(P)$ acts on convex modules in the sense below.

\begin{lemmanonum}(Lemma 31 \cite{meehan_meyer_1})
\label{hom1}
Let $P=P_n$ be totally ordered, and let $I, J$ be convex module and $\Lambda \in \mathcal{T}(P)$.  Then, $I \Lambda$ is either the zero module or convex.
\end{lemmanonum}

%%%%%%%%%%%%%%%%%%%%%%%%%%%%%%%%%%
\subsection{Bottleneck Metrics}
%%%%%%%%%%%%%%%%%%%%%%%%%%%%%%%%%%%%
\label{sec bottle}

A bottleneck metric provides an alternate metric structure on the set of isomorphism classes of $A(P_n)$-modules.  For a more general discussion of a bottleneck metric on a subcategory of a module category see \cite{meehan_meyer_1}.

The construction begins with a metric $d_2$ on the set of isomorphism classes of indecomposable $A(P_n)$-modules $\Sigma$.  Additionally, we require a function $W:{\Sigma} \to (0,\infty)$, compatible with $d_2$ in the sense that for all $\sigma_1, \sigma_2 \in \Sigma$, 
$$|W(\sigma_1) - W(\sigma_2)| \leq d_2(\sigma_1,\sigma_2).$$ 

Following \cite{induced_matchings}, \cite{zigzag},  we define a matching between two multisets $S, T$ of ${\Sigma}$ to be a bijection $f:S' \to T'$ between multisubsets $S' \subseteq S$ and $T' \subseteq T$.  For $\epsilon \in (0,\infty)$, we say a matching $f$ is an $\epsilon$-matching if the following conditions hold;
\begin{enumerate}[(i)]
\item $\textrm{for all }s \in S, W(s)>\epsilon \implies s \in S'$ 
\item $\textrm{for all }t \in T, W(t)>\epsilon \implies t \in T' $, and
\item $d_2(s,f(s) ) \leq \epsilon$, for all $s \in S$.
\end{enumerate}

The intuition is that $W$ measures the \emph{size} of an element of ${\Sigma}$, we call $W(\sigma)$ the \emph{width} of sigma.  Thus, in an $\epsilon$-matching, elements of $S$ and $T$ which are identified are within $\epsilon$, while all those not identified have width at most $\epsilon$.  

For an $A(P_n)$-module $M$, the barcode of $M$, $B(M)$, is the multiset of isomorphism classes of indecomposable summands of $M$ with their corresponding multiplicities.  Thus, $B(M)$ is precisely a multiset of elements in $\Sigma $.
\begin{definition}
Let $S, T$ be two finite multisubsets of $\Sigma$.  Suppose $d_2$ and $W$ are compatible.  Then the bottleneck distance between $S$ and $T$ is defined by, 
$$D_B(S,T) = \textrm{inf} \{\epsilon \in \mathbb{R}: \textrm{there exists and }\epsilon \textrm{-matching between  }S, T\} $$

\end{definition}
Now, if $M, N$ are $A(P_n)$-modules, we may identify $M, N$ with their barcodes $B(M), B(N)$, two multisubsets of $\Sigma$.  This identification is clearly constant on isomorphism classes.  Then, we define
$$ D_B(M,N) := D_B(B(M),B(N)).$$

While there are many examples of bottleneck metrics in the literature, in this paper, we will choose $d_2$ to be the interleaving metric corresponding to the weight $(\{a_i\},b)$.  Our width will be an algebraic analogue of the width of the support of a one-dimensional persistence module. Our definition will come from the following lemma.

\begin{lemmanonum}(Lemma 23 \cite{meehan_meyer_1})
\label{W}
Consider $P_n$ with weights given by $({\big\{a_i\big\}}_{i=1}^n,b)$, with $a_i, b >0$ in $\mathbb{R}$.  Let $I$ be convex.  Then, the following are equal;
\begin{enumerate}[(i)]
\item $W_1(I) = \textrm{min}\{\epsilon: \exists \Lambda, \Gamma \in \mathcal{T}(P), h(\Lambda), h(\Gamma) \leq \epsilon, \textrm{and Hom}(I,I\Lambda \Gamma) = 0\}$
\item $W_2(I) = \textrm{min}\{\epsilon: \exists \Lambda \in \mathcal{T}(P), h(\Lambda) \leq \epsilon, \textrm{ and Hom}(I,I {\Lambda}^2)=0\}$.
\item $W_3(I) = \textrm{min}\{\epsilon:\exists {\Lambda}_{\epsilon} \in \mathcal{T}(P) \textrm{with Hom}(I,I{{\Lambda}_{\epsilon}}^2)=0\}$.
%\item[iv.]  $\displaystyle \overline{W}(I) = {\textrm{min}}\big\{t:\textrm{max}\big\{\int\limits_{[x,t]}df,\int\limits_{[t,X+1]}df\big\}\big\}$.
\end{enumerate}
\end{lemmanonum}
For $I$ convex, set $W(I)=W_1(I)$. This will be our definition of the \emph{width} of $I$. While $W$ is defined algebraically, this definition agrees with Bauer and Lesnick in the case of one-dimensional persistence modules.

\subsection{Induced Matchings} 
\label{subsection induced}
An induced matching is a specific matching produced from an interleaving triangle. The key step in the proof of the isometry theorem of Bauer and Lesnick \cite{induced_matchings} is that an interleaving between two one-dimensional persistence modules produces an induced matching of the same height. Since our work is an algebraic analogue of the ideas of Bauer and Lesnick, this is also a key step in proving the isometry theorem in \cite{meehan_meyer_1}.

We will use the following result on injective and surjective morphisms between generalized persistence modules.

\begin{propnonum}(Theorem 4.2 \cite{induced_matchings}, Proposition 25 \cite{meehan_meyer_1})
\label{inj_surj}
Let $P=P_n$ be totally ordered and let $({\big\{a_i\big\}}_{i=1}^n,b)$ be weights.  Let $A = {\bigoplus}_iA_i, C = \bigoplus_jC_j$ be $A(P_n)$-modules.  For any module $M$, let $B(M)$ denote the barcode of $M$ viewed as a multiset, and let $\Lambda \in \mathcal{T}(P)$.  
Then,
\begin{enumerate}[(i)]
\item If $A \xhookrightarrow{f} C$ is an injection, then for all $d \in P$, the set 
\begin{eqnarray*}
&&|\{i: [-,d] \textrm{ is a maximal totally ordered subset of Supp}(A_i)\}| \leq\\
&& |\{j: [-,d] \textrm{ is a maximal totally ordered subset of Supp}(C_j)\}|, \textrm{ and }
 \end{eqnarray*}
\item If $A \xrightarrow{g} C$ is a surjection, then for all $b \in P$,
\begin{eqnarray*}
&&|\{j:[b,-] \textrm{ is a maximal totally ordered subset of Supp}(C_j) \}| \leq\\
&&|\{i:[b,-] \textrm{ is a maximal totally ordered subset of Supp}(A_i) \}|.
\end{eqnarray*}
\item If $A$ and $C$ are $(\Lambda,\Lambda)$-interleaved, and $A \xrightarrow{\phi} C\Lambda$ is one of the homomorphisms, then for all $I$ in $B(ker(\phi))$, $W(I) \leq h(\Lambda)$.
\item  If $A$ and $C$ are $(\Lambda,\Lambda)$-interleaved, and $A \xrightarrow{\phi} C\Lambda$ is one of the homomorphisms, then for all $J$ in $B(cok(\phi))$, $W(J) \leq h(\Lambda)$.
\end{enumerate}
\end{propnonum}

This was proven by Bauer and Lesnick and modified to the current context in \cite{meehan_meyer_1}. The proposition below is also in the flavor of the work of Bauer and Lesnick \cite{induced_matchings}. This is used to construct canonical injections and surjections.

\begin{propnonum}(Proposition 29 \cite{meehan_meyer_1})
\label{trim}
Let $P_n$ be totally ordered and let $({\big\{a_i\big\}}_{i=1}^n,b)$ be weights.    Let $(\phi, \psi)$ be a $(\Lambda, \Lambda)$-interleaving betweem $I$ and $M$.  Say $\phi :I\rightarrow M\Lambda$.  Then, 
\begin{enumerate}
\item[i.] $I^{-\Lambda^2}$ is a quotient of both $I$ and $im(\phi)$, and
\item[ii.]  $M^{+\Lambda^2}\Lambda$ is a submodule of both $M\Lambda$ and $im(\phi)$.
\end{enumerate}
\end{propnonum}

What follows is an adaptation of section 4 of \cite{induced_matchings}.  For any set $\Sigma'$, an \emph{enumeration} of $\Sigma'$ is a total ordering on $\Sigma'$. When $S$ is a multisubset of $\Sigma'$, an enumeration of $S$ is an total ordering on $S$ that is consistent with the enumeration of its underlying set $\Sigma'$. That is to say, for any $s\in \Sigma'$, the multiset $\{s\}\subseteq S$ is a segment in the total ordering on $S$ and the segments are arranged relative to each other according to the ordering on $\Sigma'$. Finally, for two enumerated multisubsets $S=\{s_1,\ldots,s_m\}$ and $T=\{t_1,\ldots,t_n\}$ with $m\leq n$, the \emph{canonical injection} from $S$ to $T$ is the one that sends $s_j\to t_j$ for all $1 \leq j \leq m$.

For any multisubset (that is not simply a subset), there are multiple enumerations. We will sometimes find it convenient to choose a particular enumeration (see for example, the proof of Theorem \ref{shift_isometry_theorem}). This was unnecessary in \cite{induced_matchings}.

We first define the induced matching for a surjective morphism $q:M\to N$, with $M,N\in A(P_n)$-mod. Let $\Sigma$ be the set of isomorphism classes of indecomposable $A(P_n)$-modules.
%For each $x\in P$, partition the set of indecomposables $\Sigma$ in $A(P)$-mod by \emph{lower endpoint}, $$\Sigma=\coprod_{x\in P}\Sigma_{[x,-]}=\coprod_{x\in P}\{\sigma\in\Sigma:\sigma\sim[x,y],\text{ for some }y\in P,\,y\geq x\}.$$ %or \emph{upper endpoint} $$\Sigma=\coprod_{y\in P}\Sigma_{[-,y]}=\coprod_{y\in P}\{\sigma\in\Sigma:\sigma\sim[x,y],\text{ for some }x\in P,\,y\geq x\}.$$
For every $x\in P_n$, define the set $\Sigma_{[x,-]}=\{\sigma\in\Sigma:\sigma\sim[x,y],\text{ for some }y\in P,\,y\geq x\}\subseteq\Sigma$.
Enumerate each of the sets $\Sigma_{[x,-]}$ by reverse inclusion. Now we enumerate both
$$M_{[x,-]}=\{\sigma\in B(M)\cap\Sigma_{[x,-]}\}\,\text{ and }\,N_{[x,-]}=\{\sigma\in B(N)\cap\Sigma_{[x,-]}\},$$
as multisubsets of $\Sigma'=\Sigma_{[x,-]}$. By the proposition above (Theorem 4.2 \cite{induced_matchings}, Proposition 25 \cite{meehan_meyer_1} ), as $q$ is a surjection, $|M_{[x,-]}|\geq|N_{[x,-]}|$ for all $x\in P$. Let $\Theta(q)_x$ denote the canonical injection $N_{[x,-]}\to M_{[x,-]}$ for each $x\in P$.  Clearly, the collection $\{\Sigma_{[x,-]}\}$ partition $\Sigma$ by lower endpoints.  That is, 
$$\Sigma=\bigcupdot_{x\in P}\Sigma_{[x,-]},$$ %=\bigcupdot_{x\in P}\{\sigma\in\Sigma:\sigma\sim[x,y],\text{ for some }y\in P,\,y\geq x\},$$
and so $\{M_{[x,-]}\}$ and $\{N_{[x,-]}\}$ partition the barcodes $B(M)$ and $B(N)$ respectively. Thus, we union to get the induced matching of $q$, $$\Theta(q)=\bigcupdot_{x\in P}\Theta(q)_x:B(N)\to B(M).$$

%(*Address non-uniqueness up to iso because of our defn of multi-set enumerations?)

We similarly define the induced matching for an injective morphism $i:M\to N$. Here, we use that the set $\Sigma$ is also partitioned by upper endpoints.  So
$$\Sigma=\bigcupdot_{y\in P}\Sigma_{[-,y]}=\bigcupdot_{y\in P}\{\sigma\in\Sigma:\sigma\sim[x,y],\text{ for some }x\in P,\,y\geq x\}.$$ We enumerate $M_{[-,y]},N_{[-,y]}$ as multisubsets of $\Sigma'=\Sigma_{[-,y]}$. By the proposition above, $|M_{[-,y]}|\leq|N_{[-,y]}|$ for all $y\in P$ as $i$ is an injection. Let $\Theta(i)^y$ denote the canonical injection $M_{[-,y]}\to N_{[-,y]}$, and take the disjoint union to get the induced matching of $i$, $$\Theta(i)=\bigcupdot_{y\in P}\Theta(i)^y:B(M)\to B(N).$$

Now we define the induced matching given by an interleaving triangle, as originally described by Bauer and Lesnick. Suppose $I,M$ are $(\Lambda,\Lambda)$-interleaved by $\phi,\psi$ where $\phi:I\to M\Lambda$. The morphism $\phi$ factors through its image as an injection after a surjection
$$I \xrightarrow{q_{\phi}}im(\phi) \xrightarrow{i_{\phi}} M\Lambda. $$
It is clear that the composition of induced matchings $\Theta(i_{\phi})\circ\Theta^{-1}(q_{\phi})$ is a matching $B(I)\to B(M\Lambda)$.  Since our goal is a matching between $B(I)$ and $B(M)$, we include an additional step. Note that it may be the case that $\textrm{Hom}(M,M\Lambda)$ is zero. Thus, we define a matching that does not rely on any morphism. Let $B(M,M\Lambda):B(M\Lambda)\to B(M)$ be the function which sends each $M_t\Lambda$ to $M_t$ for all $t$ such that $M_t\Lambda\neq 0$.

%**Can we still get away with this? As we have barcodes (multisubsets) written, there is no distinction between two isomorphic convex modules.**

Then, we define \emph{the induced matching $B(I)\to B(M)$ given by the triangle starting at $I$} to be the composition %$$B(M,M\Lambda)\circ\Theta(i_\phi)\circ\Theta^{-1}(q_\phi).$$

\begin{center}
\includegraphics[scale=1]{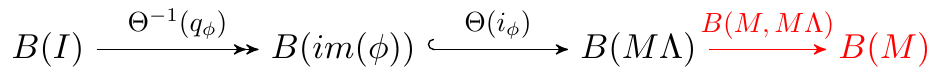}.
\end{center}

In \cite{induced_matchings}, this final step is unnecessary, as $B(M,M\Lambda)$ is a bijection of barcodes that preserves $W$. Note that while strictly speaking an induced matching depends on the enumerations for $B(I)$ and $B(M)$, different choices will produce matchings which agree at the level of isomorphism classes of indecomposables.

\section{Restriction and Inflation}

In this section we discuss restricting and inflating the module category for $A(P)$. We first make a preliminary observation.  It is easy to see that for every $m \in \mathbb{R}$ and for every $n$, there is a one-to-one correspondence between the sets
$$\{(P_n,\{a_i\}, b): (\{a_i\},b) \textrm{ are weights for }P_n\}  \xleftrightarrow{T_m} \{(X,b): X\subseteq \mathbb{R}, |X| = n \textrm{, min}(X)=m\} $$

Specifically, $T_m$ sends the tuple
$$\displaystyle\big(P_n,\{a_i\},b\big)  \xrightarrow{T_m} \big(\{m, m+a_1, m+ a_1 +a_2, ... m + \sum\limits_{i=1}^{n-1} a_i \},b\big)$$

Clearly, this assignment is invertible.  That is to say, once a left endpoint is fixed, the triple $(P_n, \{a_i\},b)$ conveys the same information as the set $\{ m + \sum\limits_{i=1}^{k} a_i\}$ plus the choice of $b$.  This is useful, as we may assume that our poset is given by the order type of the finite subset $X = \{x_1 < x_2 < ... < x_n\}$ with weights given by $a_i = x_{i+1}-x_i$ for $i \leq n-1$ and $b$.  Of course, the latter has a physical interpretation in context of the real line.  In what follows, the points in $X$ will include the jump discontinuities of the Vietoris-Rips complex of a data set. When a generalized persistence module comes from the Vietoris-Rips complex of a data set and $X$ is a superset of the jump discontinuities of the complex, we will say that ${T_m}^{-1}(X)$ is the \emph{natural choice of weights on the corresponding $P_n$}, for $n = |X|$. When this is the case, we will write $P_X$ for $P_n$ (with this choice of weights).

\label{section rest}
\begin{definition}
\label{delta X}
Let $P$ be any poset, and let $X \subseteq P$.  Suppose $I$ is a generalized persistence module for $P$ with values in the category $\mathcal{D}$.  Let $I^X$ be defined by the formulae;
\begin{itemize}
\item $I^X(x) = I(x) \in \mathcal{D}\textrm{, for all } x \in X\textrm{ , and }$
\item $I^X(x_1 \leq x_2) = I(x_1 \leq x_2) \in {\textrm{Hom}}_{\mathcal{D}}(I(x_1),I(x_2))\textrm{ for all }x_1 \leq x_2, x_1, x_2 \in X.$
\end{itemize}
\end{definition}

Then $I^X$ is a generalized persistence module for $X$ (with restricted ordering) with values in the category $\mathcal{D}$.  Moreover, it is clear that by restricting morphsims between generalized persistence modules in the obvious way we obtain a functor from ${\mathcal{D}}^P \to {\mathcal{D}}^X$.  Of particular interest will be the case when $P = \mathbb{R}$, and $X \subseteq P$ is a finite subset. When this is the case, we write $\delta^X$ for the functor
$$\delta^X: {\mathcal{D}}^{\mathbb{R}} \to {\mathcal{D}}^X .$$

We now discuss from \emph{inflating} from $A(P_X)$-mod to $A(P_Y)$-mod, when $X, Y $ are finite subsets of $\mathbb{R}$ with $X \subseteq Y$.  First, we work with translations.
\begin{definition}
\label{lambda bar}
Let $X, Y$ be finite subsets of $\mathbb{R}$, with $X \subseteq Y$.  Let $\Lambda \in \mathcal{T}(P_X)$.  Let $\bar{\Lambda}$ be given by\\
$\bar{\Lambda}(y) = \begin{cases} \textrm{max}\big\{y, \hspace{.07 in}\textrm{max}\{\Lambda x : x \in X, x \leq y\}\big\} \textrm{, if there exists }x \in X, x \leq y\\
y \textrm{ , otherwise.}
\end{cases}$
\end{definition}
Of course there is an assignment $\Lambda \to \bar{\Lambda}$ on $\mathcal{T}(P)$ for every $X, Y$ with $X \subseteq Y$.  When the context is clear, we supress the arguments $X, Y$.  The following lemma shows that $\bar{\Lambda}$ is a translation on $P_Y$ of the same height as $\Lambda$.
\begin{lemma}
\label{lambda bar}
Let $X, Y$ be finite subsets of $\mathbb{R}$, with $X \subseteq Y$.  If $\Lambda \in \mathcal{T}(P_x)$, then $\bar{\Lambda} \in \mathcal{T}(P_Y)$ and $h(\bar{\Lambda}) = h(\Lambda)$.  Moreover, ${\bar{\Lambda}}_{|X} = \Lambda$.
\end{lemma}
\begin{proof}
First, by inspection, for all $y \in Y$, $\bar{\Lambda}y \geq y$.  It is also easy to see that $\bar{\Lambda}$ restricts to $\Lambda$ on $X$ as a function.  Since its height is attained on $X$, $h(\bar{\Lambda}) = h(\Lambda)$.  Now, let $t_1, t_2 \in Y$, with $t_1 \leq t_2$.  We must show that $\bar{\Lambda}t_1 \leq \bar{\Lambda} t_2$.  Note that 
$$\textrm{max}\{\Lambda x : x \in X, x \leq y\} = \Lambda x_y \textrm{, where }x_y = \textrm{max}\{x \in X, x \leq y\}.$$

When they exist, let $x_1, x_2$ be maximal elements of $X$ less than or equal to $t_1, t_2$ respectively.  

First, suppose $t_1 \in X$, $t_2 \notin X$.  Then, 
$$t_1 \leq x_2\textrm{, so } \bar{\Lambda}t_1 = \Lambda t_1 \leq \Lambda x_2 = \bar{\Lambda} t_2. $$

On the other hand, say $t_1 \notin X, t_2 \in X$.  If $\bar{\Lambda} t_1 = t_1$, then $\bar{\Lambda}t_1 = t_1 \leq t_2 \leq \Lambda t_2 = \bar{\Lambda}t_2$.  Otherwise, $\bar{\Lambda} t_1 = \Lambda x_1.$  Then, $x_1 \leq t_1 \leq t_2$, so $\Lambda x_1 \leq \Lambda t_2$ and $\bar{\Lambda} t_1 = \Lambda x_1 \leq \Lambda t_2 = \bar{\Lambda} t_2$.  The remaining cases are handled similarly.  
\end{proof}
We now include the category $A(P_X)$-mod inside $A(P_Y)$-mod when $X \subseteq Y$ and $X, Y$ are finite subsets of $\mathbb{R}$.

\begin{definition}
\label{j(X,Y)}
Let $X, Y$ be finite subsets of $\mathbb{R}$, with $X \subseteq Y$.  Let $I \in A(P_X)$-mod.  Define $j(X,Y)I$ by the formulae;
\begin{eqnarray}
\label{vector space}
j(X,Y)I(y) = \begin{cases} I(x_y)\textrm{, }x_y \textrm{ maximal in }X, x_y \leq y\textrm{, or}\\ 0 \textrm{, if no such }x_y \textrm{ exists.} \end{cases}
\end{eqnarray}

\begin{eqnarray}
\label{linear map}
j(X,Y)I(y_1 \leq y_2) = \begin{cases} I(x_1 \leq x_2) \textrm{, where }x_i \textrm{ is maximal in }X, x_i \leq y_i\textrm{, or}\\0 \textrm{, if either of the above do not exist.}\end{cases}
\end{eqnarray}
\end{definition}

It is clear that Equations (\ref{vector space}), (\ref{linear map}) define a module for the algebra $A(P_Y)$.  Note that if $I$ is a convex module for $P_X$, $I \sim [a,b]$, then $j(X,Y)I\sim [a, b_y]$, where $b_y$ is maximal in $Y \cap [b, b^{+1})$, where $b^{+1}$ is the successor of $b$ in $X$.  That is to say, the right endpoint of the support of $I$ may move to the right.  We now discuss how morphisms in $A(P_X)$-mod can be extended to the image of $j(X,Y)$.

\begin{definition}
\label{upper Y}
Let $X, Y$ be finite subsets of $\mathbb{R}$, with $X \subseteq Y$.  Let $I, M$ be modules for $A(P_X)$, and let $\alpha \in \textrm{Hom}(I,M)$.  Let $j(X,Y)\alpha$ be defined by the formula
$$j(X,Y)\alpha(y) = \begin{cases} \alpha(x_y) \textrm{, where }x_y \textrm{ is maximal with }x_y \in X, x_y \leq y\textrm{, or }\\ \textrm{the zero homomorphism, if no such element exists.}\end{cases}$$
\end{definition}

Then, $j(X,Y)\alpha$ defines an $A(P_Y)$-module homomorphism from $j(X,Y)I$ to $j(X,Y)M$.  The proof follows from the commutativity of the diagrams below.

\begin{center}
\includegraphics[scale=1.25]{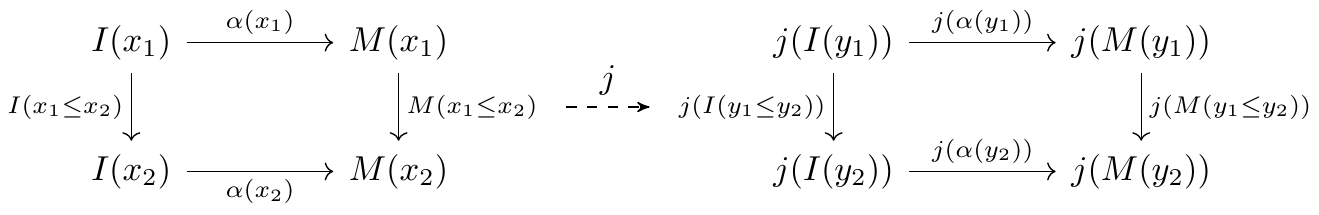}
\end{center}

\medskip
We are  now ready to compare different one-dimensional persistence modules.  Let $D_1, D_2, ... , D_n$ be finite data sets in some metric space.  Let $L_i$ be the set of jump discontinuities of the Vietoris-Rips complex of $D_i$, and let $L = L_1 \cup L_2 \cup ... \cup L_n$.  Let $\Delta(D_1,D_2, ... , D_n)$ be the collection of all finite supersets of $L$.  Clearly $L$ is a directed set under the partial ordering given by containment.  Note that for all $i$ the one-dimensional persistence modules coming from the data set $D_i$ admits the structure of an $A(P_X)$-module for any $X$ in $\Delta(D_1,D_2, ... , D_n)$.  Thus, we may compare discretized persistence modules which are a priori modules for \emph{different} poset algebras.  Since clearly any finite set of one-dimensional persistence modules can be compared in this way, from this point forward we write $\Delta(D)$ for $\Delta(D_1,D_2, ... , D_n)$.  
\begin{prop}
\label{functor}
Let $I$ be a one-dimensional persistence module that comes from data. Say $D_1,\ldots,D_n$ are such that the jump discontinuities of the Vietoris-Rips complex of $I$ are contained in the corresponding set $L$. Let $X, Y \in \Delta(D)$, with $X \subseteq Y$.  Then, $j(X,Y)$ is a fully-faithful functor from $A(P_X)$-mod to its image in $A(P_Y)$-mod.  Moreover, $j(X,Y)$ commutes with $\delta^X, \delta^Y$ in the sense that $(j(X,Y) \circ \delta^X)I = \delta^YI$.  
\end{prop}
\begin{proof}
One easily checks that $j(X,Y){(\beta \circ \alpha)} = j(X,Y)\beta \circ j(X,Y)\alpha$.  Now say $X, Y \in \Delta(D)$, where $I$ is as above.  We will show the commutativity of the triangle below.  

\begin{center}

\begin{tikzpicture}[commutative diagrams/every diagram]
	\matrix[matrix of math nodes, name=m, commutative diagrams/every cell,row sep=.7cm,column sep=1cm] {
	 \pgftransformscale{0.2}
		{} & I & {}\\
		\delta^X I & {} & \delta^Y I\\ };
		
		\path[commutative diagrams/.cd, every arrow, every label]
			(m-1-2) edge node [above] {$\delta^X$} (m-2-1)
			(m-1-2) edge node[above,xshift=.15cm] {$\delta^Y$} (m-2-3)
			(m-2-1) edge node[above] {$j(X,Y)$} (m-2-3);
		
\end{tikzpicture}

\end{center}
\noindent
First, say $I$ is convex, with $I \sim [t,T)$.  Let $y \in Y \cap [t,T)$.  Then, since $t \in X \cap Y$, there exists $x \in X$ with $x \leq y$.  Thus, let $x_y$ be maximal in $X$, with $x_y \leq y$. Since $I$ is convex, it is enough to show that $j(X,Y)\delta^XI(y)=\delta^YI(y)=K$. As $y\in[t,T)$, $t\leq x_y\leq y<T$, so $j(X,Y)\delta^XI(y)=\delta^XI(x_y)=I(x_y)=K$. Similarly, $\delta^YI(y)=I(y)=K$ as required, so the result holds for $I$ convex. The general case follows since all the above functors distribute through direct sums.  $J(X,Y)$ is fully faithful by the characterization of homomorphisms between convex modules and since $j(X,Y)$ is one-to-one on isomorphism classes of objects.
\end{proof}
Moreover, the following lemma shows that $j(X,Y)$ is compatible with the assignment $\Lambda \to \bar{\Lambda}$ (for $X, Y$).
\begin{lemma}
\label{translation and extension}
Let $X, Y$ be finite subsets of $\mathbb{R}$, with $X \subseteq Y$.  Let $\Lambda \in \mathcal{T}(P_X)$, and let $I, M$ be $A(P_X)$-modules.  Suppose $\alpha \in \textrm{Hom}(I,M)$.  Then, 
\begin{enumerate}[(i.)]
\item $j(X,Y){(I \Lambda)}= (j(X,Y)I)\bar{\Lambda}$, and 
\item $j(X,Y){(\alpha \Lambda)} = (j(X,Y)\alpha) \bar{\Lambda}$.
\end{enumerate}
\end{lemma}
\begin{proof}
Let $y\in Y$, $z_y$ the maximal element of $X$ such that $z_y\leq y$.

Then $(j(X,Y)I)\bar{\Lambda}(y)=(j(X,Y)I)\Lambda(z_y)=I(\Lambda(z_y))=(I\Lambda)(z_y)=j(X,Y)(I\Lambda)(y)$.  Similarly, $j(X,Y)\alpha\bar{\Lambda}(y)=j(X,Y)\alpha\Lambda(z_y)=\alpha\Lambda(z_y)=(\alpha\Lambda)(z_y)=j(X,Y)(\alpha\Lambda)(y)$.
\end{proof}

The results in this section are used in the next section where we prove our algebraic stability theorems.

\section{The Shift Isometry Theorem}
\label{section shift}
In this section we show that an interleaving need not produce on induced matching of the same height.  Since the existence of such a matching is the key step in the proof of an isometry theorem, this provides an obstruction to proving an isometry theorem for $A(P_n)$-mod.  We then prove a (shifted) isometry theorem by enlarging the category $A(P_n)$-mod.  We begin with an example illustrating the failure of the "matching theorem."  For a lengthier discussion, see Subsection \ref{section regu}.

\begin{example}
\label{counterexample_1}
Let $P=P_{6} = \{x_1 < ... < x_6\}$ with weights $a_1,a_3,a_4=1$ and $a_2,a_5=2$.  Let $\Lambda=\Lambda_2$ and consider the convex modules with supports depicted below.

\begin{center}
\includegraphics[scale=1]{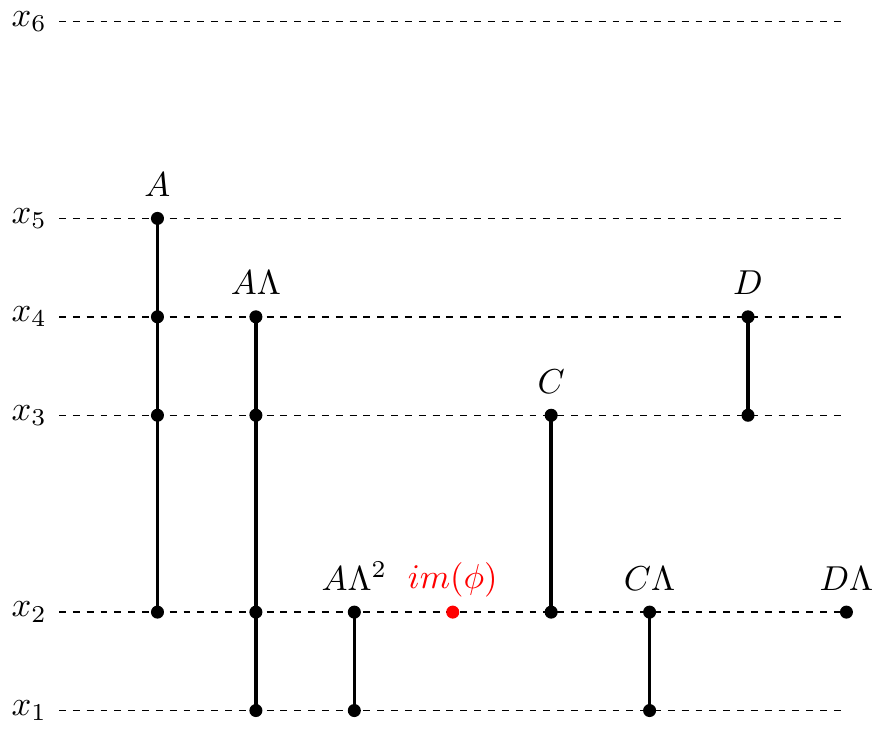}
\end{center}

Let $\phi=\Phi_{A,D\Lambda}$ and $\psi=\Phi_{D,A\Lambda}$. It is immediate that $\phi, \psi$ correspond to a $(\Lambda_2,\Lambda_2)$-interleaving. However, the induced matching pairs $A\updownarrow C$ which are not $(\Lambda_2,\Lambda_2)$-interleaved as $\Phi_{C,A\Lambda}=0$.  Thus, the induced matching corresponding to $\psi$ does generate a matching of the correct height. We can cause both induced matchings to fail by taking $I=A\oplus (C\oplus D)$ and $M=(C\oplus D)\oplus A$ to be $(\Lambda_2,\Lambda_2)$-interleaved by morphisms $\phi'=\phi\oplus\psi$ and $\psi'=\psi\oplus\phi$.

\end{example}

It is important to note that this does not say that the interleaving distance between $A, C\oplus D$ is not the bottleneck distance. In fact, they are the same. This simply says that the only known algorithm for producing a matching with the same height as the interleaving fails in this situation.

\medskip

We now work towards the proof of the shift isometry theorem.  First we show that we enlarge the category, the functor $j$ is a contraction.

\begin{prop}
\label{contraction_prop}
For $X,Y$ finite subsets of $\mathbb{R}$, $X\subseteq Y$, the functor $j(X,Y)$ is a contraction from $A(P_X)$-mod equipped with $D^X$ to its image in $A(P_Y)$-mod equipped with $D^Y$.
\end{prop}

\begin{proof}
Let $I,M\in A(P_X)$-mod. Suppose $I,M$ are $(\Lambda_\epsilon,\Lambda_\epsilon)$-interleaved in $P_X$. It suffices to show that $j(X,Y)I$ and $j(X,Y)M$ are $(\Gamma,\Gamma)$-interleaved in $P_Y$ with $h(\Gamma)\leq\epsilon$. By Lemma \ref{translation and extension}, $j=j(X,Y)$ gives the following progression of diagrams.

\begin{center}
\includegraphics[scale=1.25]{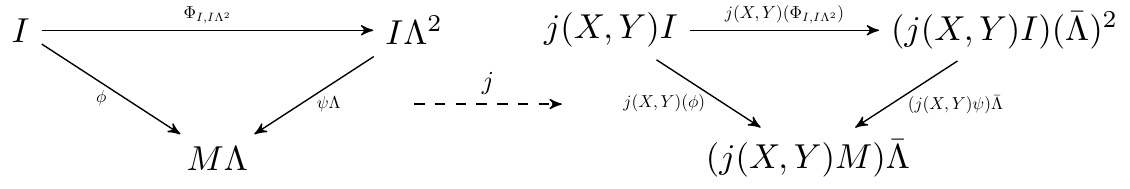}
\end{center}

The result is now obtained by letting $\Gamma=\bar{\Lambda}_\epsilon$ and noting that $h(\bar{\Lambda})=\epsilon$ by Lemma \ref{lambda bar}.

\end{proof}

We will now construct a particular refinement of $X$. This will give rise to the appropriate enlarged module category.  Let $X\subseteq\mathbb{R}$ be finite. Let $Y=X\cup\{x-\epsilon:x\in X,\epsilon\in N_1\}$ where $N_1$ denotes the set of all distances between points in $X$. Order the finite set $Y=\{y_1>y_2>\ldots>y_n\}$ by greatest to least on the real number line.

Let
$$Z_1=Y\cup\{z_{y_1}-\epsilon:\text{ for all }\epsilon\in N_1,\text{ with }z_{y_1}\text{ maximal in }Y\text{ such that }z_{y_1}<y_1\}.$$
Next, let
$$Z_2=Z_1\cup\{z_{y_2}-\epsilon:\text{ for all }\epsilon\in N_1,\text{ with }z_{y_2}\text{ maximal in }Z_1\text{ such that }z_{y_2}<y_2\}.$$
For the $i$-th step in the process, let
$$Z_{i+1}=Z_i\cup\{z_{y_{i+1}}-\epsilon:\text{ for all }\epsilon\in N_1,\text{ with }z_{y_{1+1}}\text{ maximal in }Z_i\text{ such that }z_{y_{i+1}}<y_{i+1}\}.$$
Since $Y$ is finite, the process terminates after $n$ steps. Let $Sh(X)$ be the set $Z_n$.  We will call $Sh(X)$ the \emph{shift refinement} of the set $X$.

\begin{lemma}\label{refinement}
Let $X$ be a finite subset of $\mathbb{R}$. For $q\in Sh(X)$, let $q^{+1},q^{-1}$ denote subsequent and previous elements in $Sh(X)$ respectively, where applicable.  Then, $Sh(X)$ has the property that for every $x\in X$ and $\epsilon\in N_1$, $(x-\epsilon)^{-1}-\epsilon\in Sh(X)$. Equivalently, for every $q\in Y$ and $\epsilon\in N_1$, $q^{-1}-\epsilon$ is in $Sh(X)$.
\end{lemma}

\begin{proof}
Let $q\in Y$, $\epsilon\in N_1$. For any $1\leq i\leq n$ and $k\geq i-1$, by construction the maximal element of $Z_k$ strictly less than $y_i$ is in fact precisely $y_i^{-1}\in Sh(X)$. Hence, $q=y_i$ for some $1\leq i\leq n$, and so $q^{-1}-\epsilon\in Sh(X)$.
\end{proof}

We are now ready to prove Theorem \ref{shift_isometry_theorem}.

\begin{proof}[Proof of Theorem \ref{shift_isometry_theorem}] Let $X\subseteq\mathbb{R}$ be finite and $Sh(X)$ be its shift refinement. Let $\mathcal{C}$ be the subcategory given by $im(j(X,Sh(X)))$.  First, by Proposition \ref{contraction_prop} the functor $j=j(X,Sh(X))$ is a contraction from $$(A(P_X)\text{-mod},D^X)\to(\mathcal{C},D^{Sh(X)}).$$ Thus, it suffices to show that the identity is an isometry from
 $$(\mathcal{C},D^{Sh(X)})\to(\mathcal{C},D_B^{Sh(X)}).$$

Let $I,M\in im(j(X,Sh(X)))$. Since an $\epsilon$-matching immediately produces a diagonal interleaving of the same height, $D\leq D_B$. To show the other inequality, we will prove that for any $(\Lambda,\Lambda)$-interleaving $\phi,\psi$, the induced matching of the triangle beginning at $I$ is a $ h(\Lambda)$-matching. The proof of this inequality will consist of the following three parts.
\begin{enumerate}
\item If $W(I_s)>h(\Lambda)$, then $I_s$ is matched.
\item If $W(M_t)>h(\Lambda)$, then $M_t$ is matched.
\item If $I_s$ and $M_t$ are matched with each other (independent of their $W$ values), then there is a $(\Lambda,\Lambda)$-interleaving between $I_s$ and $M_t$.
\end{enumerate}

%Our matching is a slight modification of the matching in \cite{induced_matchings}, in that it matches $B(I)$ with $B(M)$ directly rather than stopping at $B(M\Lambda)$. **This is already discussed in the preliminaries matching section, simply refer back to it?**

The proof of (1), (2) proceed as in Theorem 2 \cite{meehan_meyer_1}  with an additional consideration.  Specifically, in the present situation it is possible for non-isomorphic convex modules in $B(I)$ to be matched with isomorphic convex modules in $B(I^{-\Lambda^2})$.  Thus, we will choose a particular enumeration of each multisubset $I^{-\Lambda^2}_{[x,-]}$ of $\Sigma_{[x,-]}$ (see Subsection \ref{subsection induced}).  If $\sigma_1^{-\Lambda^2}\cong\sigma_2^{-\Lambda^2}$ in $I^{-\Lambda^2}_{[x,-]}$, then  we set $\sigma_1^{-\Lambda^2}<\sigma_2^{-\Lambda^2}$ if $\sigma_1<\sigma_2$ in $I_{[x,-]}$. This was not a concern in \cite{meehan_meyer_1}, since for the democratic choice of weights $\sigma_1\cong \sigma_2$ if and only if $\sigma_1^{-\Lambda^2}\cong \sigma_2^{-\Lambda^2}$. In our present situation, the above choice of enumeration ensures commutativity of the appropriate triangle for (1).  The proof of (2) is similar.  Note that the arguments for (1), (2) do not require any special properties of the poset $P_{Sh(X)}$.

We will now prove (3). We show that if $I_s$ and $M_t$ are matched by the induce matching, then they are $(\Lambda_\epsilon,\Lambda_\epsilon)$-interleaved. Let $I_s\sim[u,U]$, $I_s\Lambda_\epsilon\sim[w,W]$, $M_t\sim[z,Z]$, and $M_t\Lambda_\epsilon\sim[v,V]$.

If $W(I),W(M)\leq\epsilon$, then $I$ and $M$ are immediately $(\Lambda_\epsilon,\Lambda_\epsilon)$-interleaved by $\phi,\psi=0$. Assume that $W(I)>\epsilon$ or $W(M)>\epsilon$.  Then, $s\in S'$ or $t\in T'$ using the notation in Theorem 2 \cite{meehan_meyer_1}. We will show that the following morphisms constitute a $(\Lambda_\epsilon,\Lambda_\epsilon)$-interleaving of $I_s$ and $M_t$. Let $\phi'=\Phi_{I_s,M_t\Lambda}$ by the linearization of $\chi([u,V])$, and similarly $\psi'=\Phi_{M_t,I_s\Lambda}$.  Proceeding as in \cite{meehan_meyer_1} we will show that
$$w\leq z\leq W\leq Z$$ whenever $s\in S'$ or $t\in T'$.  It is enough to show that the following four statements hold.
\begin{enumerate}[(i.)]
\item If $t\in T'$, then $w\leq z$.
\item If $s\in S'$, then $z\leq W$ and $W\leq Z$.
\item If $s\in S'$ and $t\not\in T'$, then $w\leq z$.
\item If $s\not\in S'$ and $t\in T'$, then $z\leq W$ and $W\leq Z$.
\end{enumerate}

We now prove (i.) through (iv.). First, if $t\in T'$, we may define $v_0$ to be the lower endpoint of $M_t^{+\Lambda^2}\Lambda$. I.e., $v_0$ is minimal such that $\Lambda v_0\geq\Lambda^2z$. As $\Lambda(\Lambda z)\geq\Lambda^2z$, by minimality $v_0\leq\Lambda z$. Furthermore, $w$ is minimal such that $\Lambda w\geq u$. Now, as $u\leq v_0$ (by properties of induced matchings) and $v_0\leq\Lambda z$ (by above), $u\leq\Lambda z$ and so minimality of $w$ guarantees that $w\leq z$. This proves (i.).

For (ii.), note that by hypothesis $v$ is minimal such that $\Lambda v\geq z$. Therefore $v\leq u$, and so $\Lambda v\leq\Lambda u$. Combining these inequalities, $z\leq\Lambda u$. Then, $s\in S'$ guarantees that $\Lambda^2u\leq U$. As $W$ is maximal such that $\Lambda W\leq U$, we get that $\Lambda u\leq W$. Combining this with the above inequality, it follows that $z\leq W$. To prove the second inequality in (ii.), we first claim that $W=\Lambda U_0$, where $U_0$ is the maximal element such that $\Lambda^2(U_0)\leq U$.  Once this is established, we have that $\Lambda V\leq Z$ and $U_0\leq V$, so by the properties of induced matchings, we obtain $$W=\Lambda U_0\leq\Lambda V\leq Z.$$ 
Thus, let us verify that $W=\Lambda U_0$. By definition, $\Lambda X_0\leq W$. To show the opposite inequality we will first show that $W\in\mathrm{Im}\,\Lambda$. Since $[u,U]\in im(j(X,Sh(X)))$, it is immediate that $U^{+1}\in X$, and so $U^{+1}-\epsilon\in Y$.  Also $W=(U^{+1}-\epsilon)^{-1}$, since the distance from $U^{+1}-\epsilon$ to $U^{+1}$ is precisely $\epsilon$, $W\leq (U^{+1}-\epsilon)^{-1}$. Furthermore, maximality of $W$ ensures that $\Lambda W^{+1}>U$, so $W^{+1}\geq U^{+1}-\epsilon$, so $W\geq(U^{+1}-\epsilon)^{-1}$. This verifies that $W$ is precisely $(U^{+1}-\epsilon)^{-1}$. Then, by Lemma \ref{refinement}, $W-\epsilon\in Sh(X)$. As $\Lambda(W-\epsilon)=W$, we have that $W$ is in the image of $\Lambda$. Hence, as $\Lambda^2(W-\epsilon)=\Lambda W\leq U$, it must be that $U_0\geq W-\epsilon$, and so $\Lambda U_0\geq W$. Combining inequalities, $W=\Lambda U_0$, and so we have proved the desired statement. This proves (ii.).

We will now prove (iii.).  First, $I\Lambda_\epsilon=[w,W],\text{ where }w=u-\epsilon\text{, and so }\Lambda w=u.$ Since $s\in S'$ we know $\Lambda^2u\leq U$, and so $\Lambda(\Lambda^2w)\leq u$. By the maximality of $W$ under the condition $\Lambda W\leq U$, we have that $\Lambda^2w\leq W$. Finally, using (ii.) and the fact that $t\not\in T'$ guarantees that $\Lambda^2z>Z$, we have that $\Lambda^2w\leq W\leq Z<\Lambda^2z,$ so by monotonicity $w< z$. This proves (iii.).

For (iv.), one can check that $\Lambda W\leq U<\Lambda^2u\leq\Lambda^2 v_0=\Lambda^3z\leq\Lambda Z,$ so by monoticity $W<Z$.

Thus, we have shown that $\psi'=\Phi_{M_t,I_s\Lambda} \neq 0$. To finish, by Corollary \ref{nonzero_then_interleaved} in the next section, $\phi',\psi'$ comprise a $(\Lambda_\epsilon,\Lambda_\epsilon)$-interleaving between $I_s$ and $M_t$, completing the proof of (3).  This finishes the proof of Theorem \ref{shift_isometry_theorem}.  Note that the requirement that we work in $Sh(X)$ only appears in the latter half of (ii.) and in (iii.).

\end{proof}

\section{Interleaving Distance as a Limit}
\label{section interleaving}
We will now use the results from the last section to recover the classical interleaving distance as a limit.
\begin{lemma}\label{DIM_leq_delta}
Let $I, M$ be convex modules for $A(P_X)$ (with its natural metric $d$) and say $I\sim[u,U]$ and $M\sim[z,Z]$. Then, $$D(I,M)\leq\min\{\max\{W(I),W(M)\},\max\{d(u,z),d(U^{+1},Z^{+1})\}\}.$$
\end{lemma}

\begin{proof}
Let $\gamma$ denote the quantity on the right hand side above. If $\gamma=\max\{W(I),W(M)\}$, the result is obvious, since $\phi,\psi=0$ constitute a $(\Lambda_\gamma,\Lambda_\gamma)$-interleaving between $I$ and $M$.

On the other hand, suppose $\gamma<\max\{W(I),W(M)\}$.
%\end{proof}
%
%
%\begin{lemma}\label{}
%Let $I\sim[x,X]$ and $M\sim[z,Z]$ be convex modules, $\epsilon>0$, and let $\Lambda=\Lambda_{\epsilon}$. If $\Phi_{I,M%\Lambda}$ and $\Phi_{M,I\Lambda}$ are both non-zero, then $I,M$ are $(\Lambda,\Lambda)$-interleaved.
%\end{lemma}
%
%\begin{proof}
%Suppose $\Phi_{I,M\Lambda}$ and $\Phi_{M,I\Lambda}$ are non-zero. This is equivalent to $$\delta=max\{d(x,z),d(X^{+1},Z^{+1})\}\leq\epsilon.$$ It suffices to show that in fact $I,M$ are $(\Lambda_\delta,\Lambda_\delta)$-interleaved.
%
%If $\delta\geq\max\{W(I),W(M)\}$, this is immediate, as $\phi,\psi=0$ form a $(\Lambda_\delta,\Lambda_\delta)$-interleaving.
%
%Suppose $\delta<\max\{W(I),W(M)\}$.
Let $\Lambda=\Lambda_\delta$ and let $\phi=\Phi_{I,M\Lambda}$ and $\psi=\Phi_{M,I\Lambda}$. Let $I\Lambda\sim[w,W]$ and $M\Lambda\sim[v,V]$. By choice of $\gamma$, it is immediate that $v\leq u$ and $V\leq U$.  We will show that $u\leq V$. Similarly, we know that $w\leq z$ and $W\leq Z$ both hold, and will show that $z\leq W$.  We will then establish the commutativity of both interleaving triangles.

By assumption, at least one of $I, M$ has width strictly larger than $\gamma$. Suppose that $W(I)>\delta$. First, we'll show $u\leq V$, which means that $\Phi_{I,M\Lambda}$ is non-zero. Since $d(Z^{+1},U^{+1})\leq\gamma$ and $W(I)>\gamma$, it must be that $\Lambda^2 u<U^{+1}$, and so $d(\Lambda u,U^{+1})>\gamma$. This says that $\Lambda u<Z^{+1}$, i.e., $\Lambda u\leq Z$. By maximality of $V$, $u\leq V$. Hence, $\phi=\Phi_{I,M\Lambda}$, the linearization of $\chi([u,V])$ is not identically zero.

We next show that $\Phi_{M,I\Lambda}$ is non-zero. If $W(M)>\gamma$, we are done by symmetry. Thus, assume $W(M)\leq\gamma$, and $W(I)>\gamma$. Since $d(u,z)\leq\delta$ we have that $z\leq\Lambda u$. As $W(I)>\delta$, it must be that $\Lambda^2u\leq U$, so by maximality of $W$, $\Lambda u\leq W$. Hence, combining inequalities we have $z\leq W$, and so $\psi=\Phi_{M,I\Lambda} \neq0.$

Thus we have shown that when $\gamma<\max\{W(I),W(M)\}$, $\phi,\psi$ are both non-zero. It remains to show that $\phi,\psi$ give commutative interleaving triangles.

Suppose that $W(I)>\gamma$.  To show that the triangle beginning with $I$ commutes we need only show that $\psi\Lambda\circ\phi\neq0$.  By inspection, $(\psi\Lambda\circ\phi)(u) \neq 0$ as required.

By symmetry, if $W(M) > \gamma$, we are done.  Thus, we need only show the commutativity of the other triangle when $W(M) \leq \gamma$.  However, since $\mathrm{Hom}(M,M\Lambda^2)=0$, $\phi\Lambda\circ\psi=0$ as required.

Hence, $\phi, \psi$ are a $(\Lambda_\gamma,\Lambda_\gamma)$-interleaving between $I$ and $M$, so $D(I,M)\leq\gamma$.
\end{proof}

\begin{corollary}\label{nonzero_then_interleaved}
Let $\Lambda$ be a maximal translation of height $h(\Lambda)$.  Let $I, M$ be convex modules for $A(P_X)$.  Say $I\sim[u,U]$ and $M\sim[z,Z]$. If $\Phi_{I,M\Lambda}$ and $\Phi_{M,I\Lambda}$ are both non-zero, then $I,M$ are $(\Lambda,\Lambda)$-interleaved.
\end{corollary}

\begin{proof}
Let $\Phi_{I,M\Lambda}$ and $\Phi_{M,I\Lambda}$ are both non-zero only if $h(\Lambda) \geq\max\{d(u,z),d(U^{+1},Z^{+1})\}\geq\gamma.$  By Lemma \ref{DIM_leq_delta}, $\gamma \geq D(I,M)$, hence $I,M$ are $(\Lambda_\gamma,\Lambda_\gamma)$-interleaved, and so they are also $(\Lambda,\Lambda)$-interleaved as $h(\Lambda) \geq\gamma$.
\end{proof}
The next Proposition also follows from Lemma \ref{DIM_leq_delta}.
\begin{prop}
\label{distance formula}
Let $X \subseteq \mathbb{R}$ be finite, $I,M$ be indecomposables in $A(P_X)$-mod with $I\sim[u,U]$, $M\sim[z,Z]$. Then,
$$D(I,M)=\min\{\max\{W(I),W(M)\},\max\{d(u,z),d(U^{+1},Z^{+1})\}\}.$$ %where for any $q\in\mathcal{X}$, $q^{+1}$ is the minimal element of $\mathcal{X}$ that is greater than $q$.
\end{prop}

\begin{proof} By Lemma \ref{DIM_leq_delta}, we need only show that $D(I,M)\geq\gamma$. Let $\epsilon = D(I,M)$ and let $\Lambda=\Lambda_\epsilon$.  If $\phi,\psi=0$ is a $(\Lambda,\Lambda)$-interleaving between $I$ and $M$ it must be that $h(\Lambda)\geq\max\{W(I),W(M)\}$, and so $D(I,M)\geq\gamma$.

Otherwise, it must be that, without loss of generality, $\Phi_{I,I\Lambda^2}\neq0$. Hence, to have a commutative triangle beginning at $I$, it must be that $\phi,\psi$ are both non-zero. But, $\textrm{Hom}(I,M\Lambda)$ and $\textrm{Hom}(M,I\Lambda)$ are both non-zero only if $h(\Lambda)\geq\max\{d(u,z),d(U^{+1},Z^{+1})\}$, and so $D(I,M)\geq\gamma$.
\end{proof}

\medskip

We now connect our work to persistent homology.  Again, let $D_1, D_2, ... D_n$ be finite data sets in some metric space, and let $L \subseteq \mathbb{R}$ be the corresponding union of the jump discontinuities of their Vietoris-Rips complexes.  Let $X \in \Delta(D)$, and let $D^X$ denote the corresponding interleaving metric and on the category $A(P_X)$-mod.  Similarly, let $W_X$ denote the width of a convex $A(P_X)$-module.  We now work towards the proof of Theorem \ref{main theorem}, showing that the classical interleaving distance can be recovered as the limit over the directed set $\Delta(X)$.  If $I$ is a one-dimensional persistence module coming from data, we say that $I$ has endpoints in $L$ if the jump discontinuities of the Vietoris-Rips complex of $I$ are contained in $L$.

\begin{lemma}
\label{Lemma W}
Let $\sigma$ be any convex one-dimensional persistence module (for $\mathbb{R}$) whose endpoints are contained in $L$, say $\sigma \sim [r,R)$.  Then,
\begin{eqnarray*}
\displaystyle\lim\limits_{X \in \Delta(D_1,D_2)}\big(W^X(\delta^X\sigma)\big)=W(\sigma) = \frac{R-r}{2}.
\end{eqnarray*}
\end{lemma}
\begin{proof}
Let $\epsilon > 0 \in \mathbb{R}$.  Let $Y = Y(\epsilon)$ be any element of $\Delta(D)$ such that
\begin{enumerate}[(i.)]
\item the max$(Y) > \textrm{max}(L)$, and 
\item the difference between consecutive elements of $Y \cap [m,M+\epsilon]$ is less than $\frac{1}{2}\epsilon$.
\end{enumerate}

Note that any superset of $Y$ necessarily satisfies (i.), (ii.) as well.  Let $X' \in \Delta(D)$ with $X' > Y$.  Then,

$$W^{X'}(\delta^{X'}\sigma) = \textrm{min}\big\{\big\{ \textrm{max}\big\{x - r, R-x\big\} : x \in X' \cap [r,R)\big\}.$$

\medskip
Since $\frac{R-r}{2}$ must be within $\frac{1}{2}\epsilon$ of some $x$, clearly 
$$\big| W^{X'}(\delta^{X'}\sigma)) - W(\sigma)\big| < \epsilon \textrm{ as required.}$$
\end{proof}

We point out that condition (i.) above removes the consideration of the weight "$b$" from the discussion.  Next we will show that if $\sigma, \tau $ are convex one-dimensional persistence modules (for $\mathbb{R}$) then their interleaving distance can be recovered as a discrete limit as well.  This establishes Theorem \ref{main theorem} for convex modules.

\begin{lemma}
\label{lemma D}
Let $\sigma, \tau$ be any convex one-dimensional persistence modules whose endpoints are contained in $L$.  Say $\sigma \sim [r,R), \tau \sim [s,S)$.  Then,
$$\displaystyle\lim\limits_{X \in \Delta(D)}\big(D^X(\delta^X\sigma,\delta^X\tau)\big) =  D(\sigma,\tau).$$

\begin{proof}
Proceeding as in the proof of Lemma \ref{lemma D}, let $\epsilon$ be positive and set $Y= Y(\epsilon)\in \Delta(D)$.  Let $X' \in\Delta(D)$ with $X' > Y$.  Then, by Proposition \ref{distance formula},

$$D^{X'}(\delta^{X'}\sigma,\delta^{X'}\tau) = \textrm{min}\Big\{ \hspace{.1 in}\textrm{max}\{ W^{X'}(\delta^{X'}\sigma),W^{X'}(\delta^{X'}\tau)\}, \hspace{.1 in}\textrm{max}\{|r-s| ,|R - S| \}  \Big\}.$$

Clearly, this is within $\epsilon$ of 

$$D(\sigma,\tau) = \textrm{min}\Big\{ \hspace{.1 in}\textrm{max}\{ W(\sigma),W(\tau)\}, \hspace{.1 in}\textrm{max}\{|r-s| ,|R-S| \}  \Big\},$$

by Lemma \ref{Lemma W}.  The result follows.

\end{proof}
\end{lemma}
Now that we have established Lemmas \ref{Lemma W}, \ref{lemma D}, we are ready to prove Theorem \ref{main theorem}

\begin{proof}
Let $I, M$ be one-dimensional persistence modules whose endpoints lie in $L$.  Let $\gamma > 0$, and let $\epsilon$ be such that
$$\epsilon - \frac{1}{2}\gamma < D(I,M) \leq \epsilon $$ 

Let $Y \in \Delta(D)$ be such that
\begin{enumerate}[(i.)]
\item max$(Y) > \textrm{max}(L)$, and
\item $y_{i+1}-y_i < \frac{1}{8}\gamma$, for $y_i \in Y\cap [m,M+\frac{1}{4}\delta]$
\end{enumerate}

\medskip
Let $X' \in \Delta(D)$, with $X' \supseteq Y$.  We will show that 
$$D(I,M) - \delta < D^{X'}(\delta^{X'}I,\delta^{X'}M) \leq D(I,M) \leq D(I,M) + \gamma $$

First, let $\delta^{X'}\sigma \in B(\delta^{X'}I) \cup B(\delta^{X'}M)$, and say $W^{X'}(\delta^{X'}\sigma) > \epsilon + \frac{1}{2}\gamma$.  Since $\epsilon \geq D(I,M)$, there exists a $(\Lambda_{\epsilon}, \Lambda_{\epsilon})$-interleaving between $I$ and $M$.  But then, by the isometry theorem of Bauer and Lesnick (\cite{induced_matchings}), there is an $\epsilon$-matching between $B(I)$ and $B(M)$.  But then, we have that
\begin{eqnarray*}
W(\sigma) + \frac{1}{2}\gamma > W^{X'}(\delta^{X'}\sigma) > \epsilon + \frac{1}{2}\gamma \textrm{, so } W(\sigma) > \epsilon.
\end{eqnarray*}

Therefore $\sigma$ is matched with some $\tau$ in the opposite barcode.  So by hypothesis,
\begin{eqnarray*}
D(\sigma,\tau) = \textrm{min}\Big\{ \hspace{.1 in}\textrm{max}\{ W(\sigma),W(\tau)\}, \hspace{.1 in}\textrm{max}\{|r-s| ,|R-S| \}  \Big\}
\end{eqnarray*}

where $\sigma \sim [r,R)$ and $\tau \sim [s,S)$.

But,
\begin{eqnarray*}
&&\textrm{max}\big\{  W^{X'}(\delta^{X'}\sigma),W^{X'}(\delta^{X'}\tau) \big\} \leq \textrm{max}\big\{ W(\sigma)+\frac{1}{2}\gamma, W(\tau) + \frac{1}{2}\gamma \big\} = \\
&&\textrm{max}\big\{W(\sigma),W(\tau) \big\} + \frac{1}{2}\gamma.
\end{eqnarray*}

Therefore,
$$D^{X'}(\delta^{X'}\sigma, \delta^{X'}\tau) \leq D(\sigma, \tau) + \frac{1}{4}\gamma \leq \epsilon + \frac{1}{2}\gamma < D(I,M) + \gamma. $$

So, the assignment $\delta^{X'}\sigma \updownarrow \delta^{X'}\tau \iff \sigma \updownarrow \tau$ defines a diagonal interleaving (a matching) between $\delta^{X'}I$ and $\delta^{X'}M$ of height $\epsilon + \frac{1}{2}\gamma$.  Thus, $D^X(\delta^{X'}I,\delta^{X'}M) \leq D(I,M) + \gamma$ as required.
\vspace{.2 in}

Now, if $D^{X'}(\delta^{X'}I,\delta^{X'}M) \geq D(I,M) - \gamma$ we are done.  Thus for a contradiction suppose that $D^{X'}(\delta^{X'}I,\delta^{X'}M) < D(I,M) - \gamma$.  Then, there exists a weight $\epsilon$ for $X'$ with $\epsilon < D(I,M) - \gamma$ and there exists a $(\Lambda^{X'}_{\epsilon},\Lambda^{X'}_{\epsilon})$-interleaving between $\delta^{X'}I$ and $\delta^{X'}M$, where $\Lambda^{X'}_{\epsilon}$ is the maximal translation of height $\epsilon$ for $P_{X'}$.  Then, by Theorem \ref{shift_isometry_theorem} there is an $\epsilon$-matching  between $B(j(X',Z)\delta^{X'}I)$ and $B(j(X',Z)\delta^{X'}M)$ for $Z = Sh(X')$.  By Proposition \ref{functor}, 
$$ j(X',Z)\delta^X I = \delta^Z I \textrm{ and }  j(X',Z)\delta^X M= \delta^Z M .$$
Thus, there is an $\epsilon$-matching between $B(\delta^ZI)$ and $B(\delta^ZM)$.  Now, say $W(\sigma) > \epsilon + \frac{1}{2}\gamma$.  Then,
\begin{eqnarray*}
W^Z(\delta^Z\sigma) \geq W(\sigma) > \epsilon + \frac{1}{2}\gamma  > \epsilon.
\end{eqnarray*}

Therefore, $\sigma^Z \updownarrow \tau^Z$ for some element $\delta^Z\tau$ of the opposite barcode, with 
$$D^Z(\sigma^Z,\tau^Z) \leq \epsilon < D(I,M) - \gamma. $$

Therefore, define the matching $\sigma \updownarrow \tau \iff \sigma^Z \updownarrow \tau^Z$.  But then, 
$$D(\sigma, \tau) \leq D^Z(\sigma^Z, \tau^Z) + \frac{1}{2}\gamma \leq \epsilon +\frac{1}{2}\gamma < D(I,M) - \frac{1}{2}\gamma. $$

This matching shows that 

$$D(I,M) \leq \epsilon + \frac{1}{2}\gamma  < D(I,M) - \frac{1}{2}\gamma .$$

As this is clearly a contradiction, so it must be the case that $D^{X'}(\delta^{X'}I,\delta^{X'}M) \geq D(I,M) - \gamma$ as required.  The result is proven.
\end{proof}

\section{Regularity}
\label{section regu}

As we have seen, the poset $P_n$ with arbitrary choice of weights has the property that an interleaving between two modules need not produce an induced matching of barcodes of the same height (see Example \ref{counterexample_1}).  While this does not necessarily mean that the isometry theorem is false in this context, it is clearly an obstruction to its proof.  In this section, we show that when we identify $(P_n,\{a_i\}, b)$ with $P_X$, $X \subseteq \mathbb{R}$, unless $X$ satisfies certain regularity conditions there will always exist interleavings whose induced matchings do not have the same height.  One would not expect such regularity for a poset $P_X$ which comes from real world data.

In what follows, it is convenient to work with maximal translations. We will define what it means for a poset to be \emph{regular} after examining some conditions on maximal translations.

Let $x_i<x_l$ in $X$, and let $\Lambda=\Lambda_{(x_l-x_i)}$.  Suppose that
\begin{enumerate}[(a)]
\item $x_{l+1}<\Lambda(x_{i+1})$,
\item $\Lambda(x_{l+1})<\Lambda^2(x_{i+1})$, and
\item $\Lambda(x_{i-1})>x_{i-1}$.
\end{enumerate}
Then, one can produce an interleaving whose induced matching has strictly larger height. If $X$ is identically the set of jump discontinuities of a data set, one would expect the the existence of some $x_i < x_l$ satisfying the above.  

On the other hand, if $X$ avoids conditions (a) or (b) for all $x_i<x_l$ we say that $X$ will be regular.  Property (c) is a purely technical condition that will not be commented on further.  Roughly speaking, a regular set has a periodicity associated both with its elements and the spaces between its elements.  We now define regularity.  After the definition, we connect regularity to the absence of a maximal translation of the form above.

\begin{definition}
Let $X$ be a finite subset of $\mathbb{R}$, and let $x_i \in X = \{x_1< x_2 < ...< x_n\}$.  Then, $X$ is regular at $x_i$ if for all $x_l>x_i$, either
\begin{enumerate}[(i.)]
\item $x_{l+2}-x_l>x_{i+1}-x_i$, or
\item $x_{l+2}-x_l\leq x_{i+1}-x_i$, and $x_{k+1} > x_{l+t}+x_l-x_i$, where $t$ be maximal such that $x_{l+t}-x_l\leq x_{i+1}-x_i$, and $\Lambda_{(x_l-x_i)}(x_{l+1})=x_k$.
\end{enumerate}
We say that the set $X$ is regular if $X$ is regular at every $x_i\in X$.
\end{definition}

We now explain regularity at $x_i$.  Let $\epsilon=x_l-x_i$ and $\Lambda=\Lambda_\epsilon$, and note that $\Lambda(x_i)=x_l$ and $\Lambda(x_{l+1})=x_k$. In addition, by the choice of $t$ it is always the case that $\Lambda(x_{i+1})=x_{l+t}$.  Clearly, if $x_i$ is regular and $x_i<x_l$, exactly one of (i.), (ii.) hold.

First, if (i.) holds at $x_l$ the spacing of points in the poset $X$ is uniform in the sense that the length of the edge from $x_i$ to $x_{i+1}$ is surpassed by the sum of the two consecutive edges $x_l$ to $x_{l+2}$. In terms of the translation $\Lambda$, property (i.) says that $0\leq t<2$ or $\Lambda(x_{i+1})\leq x_{l+1}$. This is the negation of (a) above.  On the other hand, if (ii.) holds at $x_l$, the "hole'' in $X$ given by the edge from $x_i$ to $x_{i+1}$ is periodic.  Specifically, there are no vertices in $X$ contained in the real interval $(x_{l+1}+\epsilon,x_{l+t}+\epsilon]$ (see the figure below).  In terms of the translation $\Lambda$, property (ii.) corresponds to the statement that $\Lambda(x_{l+1})=\Lambda(x_{l+t})$.  Of course, this is the negation of (b) above.
\begin{center}
\includegraphics{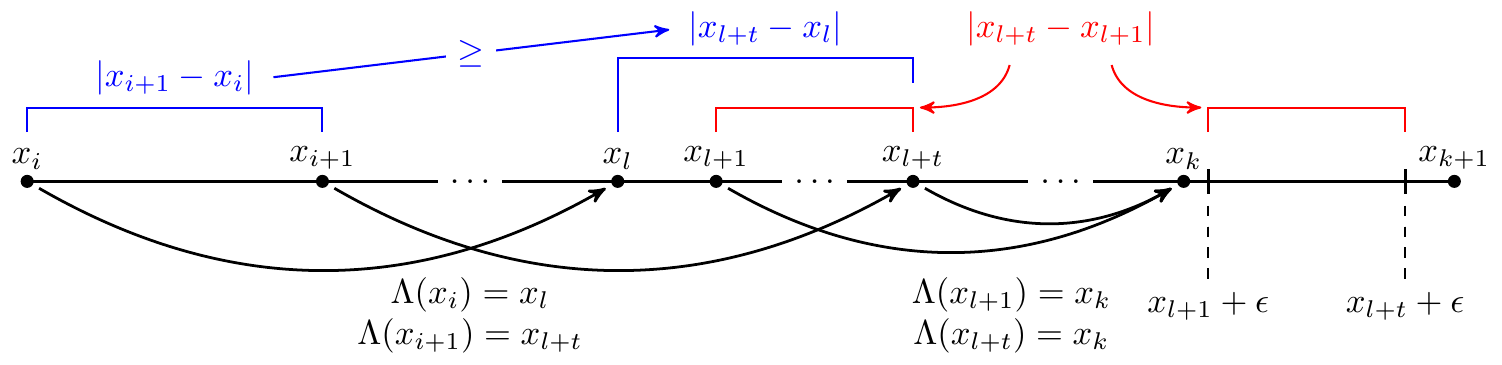}
\end{center}

Of course, if $X$ is regular, it is regular at every $x_i$.  Of particular interest is the case in which $x_{i+1}-x_i$ is large. Then, for all $x_l$ with $x_i<x_l$ where (i.) is satisfied, the sum of the two edges after $x_l$ must be long. Thus, since $x_{i+1}-x_i$ is large, the distances $x_{l+2}-x_{l+1}, x_{l+1}-x_l$ taken together must also be large.  Alternatively, if (ii.) is satisfied then a hole close to the size of $x_{i+1}-x_i$ must be repeated at a distance of exactly $\epsilon$ away from $x_{l+1}$.  Since the distance $x_{i+1}-x_i$ is large, this says that large holes must be repeated regularly.  We emphasize that the above statements must hold for all $x_i \in X$ if $X$ is regular.  Alternatively, if $X$ fails to be  regular (with an addition technical condition), then there always exist interleavings whose corresponding induced matchings have strictly larger heights.

\begin{prop}
If $X$ is not regular at some $x_i<x_l$ where $x_{i-1}$ is not fixed by $\Lambda_{(x_l-x_i)}$, then there exists an interleaving whose induced matching has strictly larger height (see Example \ref{counterexample_1}).
\end{prop}

\begin{proof}
By the above remarks, let $x_i<x_l$ be such that the translation $\Lambda=\Lambda_{(x_l-x_i)}$ has the properties $x_{l+1}<\Lambda(x_{i+1})$, $\Lambda(x_{l+1})<\Lambda(x_{l+t})$ and $x_{i-1}$ is not fixed by $\Lambda$.  Let $\epsilon=x_l-x_i$, and let $\Lambda=\Lambda_\epsilon$. Consider the following convex modules, $A\sim[x_i,\Lambda(x_{l+1})],$ $C\sim[x_i,x_l]$, and $D\sim[x_l,x_{l+t'}],$ where $1\leq t'<t$ is maximal such that $\Lambda(x_{l+t'})=\Lambda(x_{l+1})$. Note that the vertex $x_{l+t'}$ is also the upper endpoint of $A\Lambda$.

We then define the $(\Lambda_\epsilon,\Lambda_\epsilon)$-interleaving between $A$ and $C\oplus D$ by the diagonal morphisms $\phi=\Phi_{A,D\Lambda}$, $\psi=\Phi_{D,A\Lambda}$. One easily checks that this is indeed an interleaving. However, the induced matching corresponding to the triangle beginning at $\phi$ matches $A$ with $C$.  Clearly $A$ and $C$ are not $(\Lambda,\Lambda)$-interleaved, as $W(A)>\epsilon$ but $\Phi_{C,A\Lambda}=0$.  Proceeding as in Example \ref{counterexample_1} by setting $I = A \oplus (C\oplus D), M = (C \oplus D) \oplus A$ with $\phi' = \phi \oplus \psi$ and $\psi' = \psi \oplus \phi$ we produce an interleaving where both induced matchings have strictly larger height.
\end{proof}

\begin{center}
\includegraphics[scale=1]{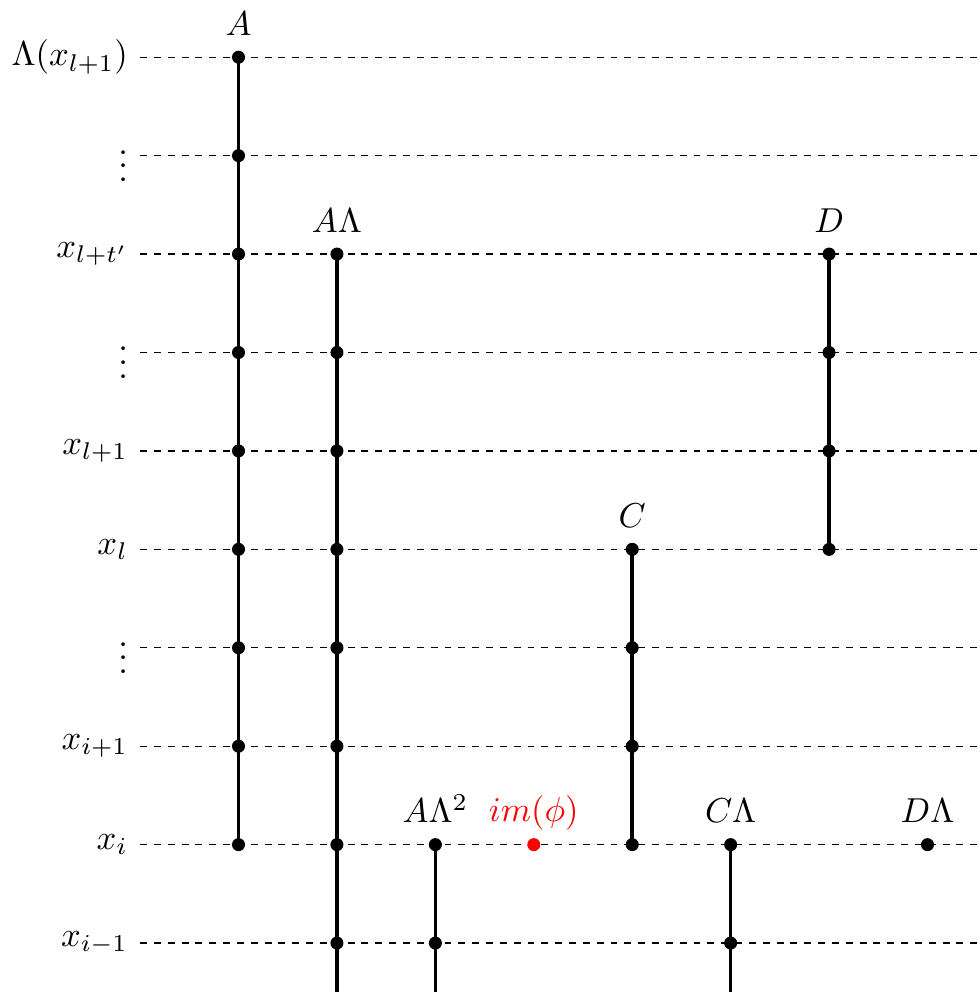}
\end{center}

This analysis shows that proof of the "matching theorem" is likely to fail for the poset $P_X$.  Therefore, it is necessity (at this point) to enlarge the category to obtain an isometry on $A(P_X)$-mod in the sense of Theorem \ref{shift_isometry_theorem}.

\vspace{1 in}

\bibliography{master_bib}
\bibliographystyle{alpha}

\end{document}